\documentclass[a4paper,11pt]{article}
\usepackage{amsmath}
\usepackage{amsfonts}
\usepackage[latin9]{inputenc}
\usepackage[english]{varioref}
\usepackage{dcolumn}
\usepackage[height=22cm , width = 16cm , top = 3cm , left = 3cm, a4paper]{geometry}
\usepackage[a4paper]{geometry}
\usepackage[final]{graphicx}
\usepackage{epsfig}
\usepackage{pstricks}
\usepackage{psfrag}
\usepackage{rotating}
\usepackage{booktabs}
\usepackage{delarray}
\usepackage{amssymb}
\usepackage{rotating}
\usepackage{subfigure}
\usepackage{layout}
\usepackage{amsthm}

				\newtheorem{thm}{Theorem}[section]
				
				\newtheorem{lem}[thm]{Lemma}
				\newtheorem{propos}[thm]{Proposition}

				\newcommand{\enter}{\bigskip}

\begin{document}
\thispagestyle{empty}
\author{Sanjiv Kumar Bariwal ${}^1$	\footnote{{\it{${}$ Corresponding author. \, Email address:}} p20190043@pilani.bits-pilani.ac.in}	, Rajesh Kumar ${}^2$\footnote{{\it{${}$ Email address:}} rajesh.kumar@pilani.bits-pilani.ac.in}\\
\footnotesize ${}^{1,2}$Department of Mathematics, Birla Institute of Technology and Science Pilani,\\ \small{ Pilani Campus, Rajasthan-333031, India}\\
}
	\date{}																		
\title{{ Weak convergence analysis for non-linear collisional induced breakage equation with singular kernel}}	
									
\maketitle
										
\begin{quote}
{\small {\em\bf Abstract}}: The phenomenon of collisional breakage in particulate processes has garnered significant interest due to its wide-ranging applications in fields such as milling, astrophysics, and disk formation. This study investigates the analysis of the pure collisional breakage equation (CBE), characterized by its nonlinear nature with presence of locally bounded collision kernels and singular breakage kernels. Employing a finite volume scheme (FVS), we discretize the continuous equation and investigate the weak convergence of the approximated solution of the conservative scheme towards the continuous solution of CBE. A weight function is introduced to ensure the conservation of the scheme. The non-negativity of the approximated solutions is also shown with the assistance of the mathematical induction approach. Our approach relies on the weak 
$L^1$ compactness argument, complemented by introducing a stable condition on the time step.

 \end{quote}
 
{\bf{Mathematics Subject Classification (2020) 45L05, 45K05, 65R10}}\\

\section{Introduction}
Particulate processes like aggregation (coagulation) and fragmentation (breakage) play a crucial role in the dynamics of particle evolution, illustrating how particles may either combine to form larger aggregates or break apart into smaller fragments. A particle is typically characterized by a single size variable, such as its volume or mass.  The non-linear nature of the breakage process arises from collisions between parent particles and is often referred to as the non-linear collision-induced breakage equation \cite{cheng1988scaling}. This nonlinearity lends itself to significant applications across various fields such as chemical engineering, cloud formation, and planet formation. In scenarios where a cluster breaks apart spontaneously into smaller clusters, it is described by the linear breakage equation. However, in the case of collisions between clusters, some of the matter from one cluster can be transferred to the other, leading to the creation of larger clusters. This phenomenon exemplifies the non-linear collision dynamics between particles.
Non-linear models are integral to various fields, such as the simulation of milling and crushing processes, the distribution patterns of asteroids, and the behavior of particles in fluidized beds, see \cite{chen2020collision,yano2016explosive,lee2017development}.\\

The concentration function  $c(\varsigma,\epsilon)\geq 0$ of particles of volume $\epsilon \in ]0,\infty[$ at time $\varsigma\geq 0$ is portrayed by the following non-linear CBE
\begin{align}\label{maineq}
\frac{\partial{c(\varsigma,\epsilon)}}{\partial \varsigma}=  \int_0^\infty\int_{\epsilon}^{\infty} K(\rho,\sigma)b(\epsilon,\rho,\sigma)c(\varsigma,\rho)c(\varsigma,\sigma)\,d\rho\,d\sigma -\int_{0}^{\infty}K(\epsilon,\rho)c(\varsigma,\epsilon)c(\varsigma,\rho)\,d\rho
\end{align}
with initial datum
\begin{align}\label{initial}
c(0,\epsilon)\ \ = \ \ c^{in}(\epsilon) \geq 0.
\end{align}
The first integral entity in Eq.(\ref{maineq}) represents the creation of particles with volume $ \epsilon$ resulting from the collisional breakage between particles with volumes $\rho$ and 
$\sigma$, commonly referred to  the birth term. The second integral component, known as the death term, accounts for the reduction of particles with volume $\epsilon$ due to interactions with particles of volume $\rho$. In Eq.(\ref{maineq}), the collision kernel $K(\epsilon,\rho)$ quantifies the frequency of collisions between particles of volumes $\epsilon$ and $\rho$. Typically, $K(\epsilon,\rho)$ is both non-negative and symmetric, implying that $K(\epsilon,\rho) = K(\rho,\epsilon)$. The breakage distribution function $b(\epsilon,\rho,\sigma)$ delineates the rate at which particles of volume $\epsilon$ are generated from the fragmentation of a particle of volume $\rho$ as a result of its collision with a particle of volume $\sigma$. The function $b$ holds
$$b(\epsilon,\rho,\sigma)\neq 0\,\,\, \text{for} \hspace{0.4cm} \epsilon \in (0,\rho)\,\,\, \text{and} \hspace{0.4cm} b(\epsilon,\rho,\sigma)=0 \,\,\,\text{for} \hspace{0.4cm} \epsilon> \rho,$$
 and volume conservation law 
 \begin{align}\label{breakageconservation}
 \int_{0}^{\rho}\epsilon b(\epsilon,\rho,\sigma)d\epsilon=\rho,
 \end{align}
 with the total particles after breakage of particle size $\rho$
 \begin{align}\label{breakageparticle}
  \int_{0}^{\rho}b(\epsilon,\rho,\sigma)d\epsilon=\mu (\rho)\leq N <\infty.
  \end{align}
  Eq.(\ref{breakageconservation}) has  the property of the local volume conservation during each collisional breakage event (total volume of the daughter particles of size $\epsilon$ is equal to size of parent particle $\rho$). 
Let us define the $i^{th}$ moment for the solution of the Eq.(\ref{maineq}) as
\begin{align}\label{moment}
M_{j}(\varsigma)=\int_{0}^{\infty}{\epsilon}^{j}c(\varsigma,\epsilon)\,d\epsilon.
\end{align}
In the above  Eq.(\ref{moment}), the zeroth $(j=0)$ and first  $(j=1)$ moments stand for the total number of particles in the system and its total volume, respectively.
  The assumption  $M_{1}^{in} < \infty$ is taken which is the initial particle volume in the closed particulate system.\\
  
  \section{Literature and Motivation}
Let's review the literature on the linear breakage equation, focusing on the studies that have explored the existence and uniqueness of solutions and the derivation of similarity solutions for various simplified breakage kernels, see \cite{ziff1991new,ziff1985kinetics,peterson1986similarity,breschi2017note} and the references therein. To solve the linear breakage equation for different kinds of physical breakage kernels, several numerical schemes were developed: Monte-carlo \cite{chen2022reconstruction}, sectional method \cite{kostoglou2009sectional}, method of moments \cite{falola2013extended}, finite volume method (FVM) \cite{kumar2013numerical,singh2022finite} and references therein.


Further, the authors \cite{bourgade2008convergence} investigated the weak convergence analysis (based on weak $L^1$ compactness technique) of FVS for  binary coagulation-breakage equation for locally bounded kernels and error estimation of approximated solution for uniform meshes. The analysis has been extended for binary coagulation and multiple breakage equation having the singular breakage kernel using the weighted FVM discretization scheme, see \cite{bariwal2023convergence}. The mathematical results on stability and convergence analysis of FVM for coupled coagulation and multiple breakage equation are discovered and numerically verified using four different kind of meshes in \cite{Rajesh2014}. \\

The pioneering research conducted by Cheng and Redner \cite{cheng1988scaling} delved into scaling solutions of the rate equation, elucidating the distinctions between linear breakage \cite{ziff1985kinetics} and non-linear (\ref{maineq}) processes. 
In their investigation, the researchers employed a widely recognized scaling solution for the concentration function, given by $c(\varsigma, \epsilon) \sim s^{-2} \varphi(\epsilon / s)$, where $s$ denotes the characteristic cluster mass, defined as $s = \varsigma^{-1 / \lambda}$ for $\lambda > 0$. The study by Cheng \cite{cheng1990kinetics} examined the asymptotic behavior of models where a two-particle collision results in either: (1) both particles fragmenting into two equal parts, (2) only the larger particle splitting into two, or (3) only the smaller particle breaking apart. Subsequently, Krapivsky  \cite{krapivsky2003shattering} explored the shattering phenomenon, particularly focusing on scenario (2), where particles transform into dust particles due to a discontinuous transition. Additionally, scenario (3) involves a continuous transition with the dust mass increasing steadily from the fragments.
 The study by Kostoglou et al. \cite{kostoglou2000study} provided analytical solutions for two specific scenarios: $K(\epsilon,\rho) = 1$ with $b(\epsilon,\rho, \sigma) = 2/\rho$, and $K(\epsilon,\rho) = \epsilon\rho$ with $b(\epsilon,\rho, \sigma) = 2/\rho$, using a mono-disperse initial condition $\delta(x-1)$. They also explored self-similar solutions for the sum kernel $K(\epsilon,\rho) = {\epsilon}^{\omega} + {\rho}^{\omega}$ and the product kernel $K(\epsilon,\rho) = {\epsilon}^{\omega} {\rho}^{\omega}$, where $\omega > 0$. Additionally, Barik et al. \cite{barik2020global} examined the existence of classical solutions  of coagulation and CBE under collision kernels that increase indefinitely for large volumes and a binary breakage distribution function. In a related study, Giri et al. \cite{giri2021existence} discussed the existence and non-existence of mass-conserving weak solutions for specific collision kernels.
 
In their study, Giri et al. \cite{giri2021weak} explored the global existence of solutions for collision kernels of the form $K(\epsilon,\rho)={\epsilon}^{\alpha}{\rho}^{\beta}+{\epsilon}^{\beta}{\rho}^{\alpha}$. They focused on kernels exhibiting singularities when $\alpha <0$ and those that are locally bounded when $\alpha=0$ for small volumes. Few series form schemes called semi-analytical methods are implemented on CBE to compute the truncated approximate solutions for some physical kernels, and convergence analysis of the series solutions was executed for a particular set of kernels, see \cite{bariwal2024non,hussain2024collisional}.  \\

Currently, in the numerical sense, various researchers have identified the FVM as a highly effective numerical technique for addressing aggregation and breakage equations, largely due to its ability to conserve mass accurately. Therefore, FVM is presented for collision-induced breakage equation to analyze the concentration function.


 As far as the author knows, previous research has yet to consider the weak convergence of the numerical approach (FVM) used to solve the CBE. In order to solve the model with locally bounded collision and a class of singular breakage kernels, i.e., breakage functions are attaining the singularity near zero, this work attempts to analyze the weak convergence of the FVM for the non-uniform mesh. Under a stability condition on time step, we prove the non-negativity, equiboundedess and equiintegrability of the sequence of the approximated solutions. The weak $L^1$ compactness technique is used in the proof.\\

Prioritizing the functional setting is the first step towards moving forward, considering that the initial mass conservation attribute in (\ref{moment}) is required. At time $\varsigma=0$, the total number of particles in the system must be finite. Consequently, we create the solution space that shows how the discretized numerical solution of the CBE converges to the weak solution of equation (\ref{maineq}), which is identified as a  weighted $L^1$ space that was not considered in \cite{bourgade2008convergence}. 
Let us suppose that the initial function $c^{in}$ satisfies    \begin{align}
c^{in}\in \mathbb{S}^{+},
\end{align}
where $\mathbb{S}^{+}$ is the positive cone of the Banach space 
	\begin{align*}
		\mathbb{S} = \{c\in L^{1}(\mathbb{R}^{+})\cap L^1(\mathbb{R}^{+}, \Big(\epsilon^r+\frac{1}{\epsilon^{2p}}\Big)\,d\epsilon): c\geq 0, \|c\|< \infty\},
	\end{align*}
consider the expression $\|c\|= \int_{0}^{\infty}\Big(\epsilon^r+\frac{1}{\epsilon^{2p}}\Big)c(\epsilon)\,d\epsilon,$ $ p\geq 0, r \in[1,u)$, with $u$ being a finite value. The term $L^1({\mathbb{R}}^{+}, \Big(\epsilon^r+\frac{1}{\epsilon^{2p}}\Big)d\epsilon)$ denotes the space of Lebesgue measurable real-valued functions on $\mathbb{R}^{+}$ that are integrable with respect to the measure $\Big(\epsilon^r+\frac{1}{\epsilon^{2p}}\Big)\, d\epsilon.$\\
	
Next, the properties of the collisional kernel $K$
 and the breakage distribution function $b$ are defined as follows:  There exist  $\zeta, \eta$  with $0<\zeta \leq \eta \leq 1, $ and $\alpha\geq 0$ such that 	
\begin{align}\label{Collisional func}
H1: \hspace{0.2cm} K(\epsilon,\rho)=
\alpha(\epsilon^{\zeta}\rho^{\eta}+\epsilon^{\eta}\rho^{\zeta}),  \quad (\epsilon,\rho)\in ]0,\infty[\times ]0,\infty[. 
\end{align}
\begin{align}\label{breakage funcn}
H2:  \hspace{0.2cm}  \int_{0}^{\rho}\epsilon^{-\upsilon p}{b(\epsilon,\rho,\sigma)}^{\tau}d\epsilon \leq Q\rho^{-\upsilon p+1-\tau}, \upsilon >0, \tau \in[1,2),
\end{align}
where $Q$ is a positive constant and integral (\ref{breakage funcn}) contains  the breakage kernels having the sinularity near to zero. Few examples those holds Eq.(\ref{breakage funcn}), i.e.,  $b(\epsilon,\rho,\sigma)=\frac{(h+1){\epsilon}^{h}}{{\rho}^{1+h}}$, for $-1<h\leq 0$ and with relation of $\tau=1,\,\, \upsilon p-h<1$.\\

The structure of this article is as follows: Section \ref{scheme} outlines the discretization technique based on the FVM and the volume-conservative form of the fully discretized CBE. Next, Section \ref{convergence} probes into an in-depth convergence analysis. Finally, Section \ref{conclusion} provides a summary and concluding remarks.
\section{Conservative Numerical Scheme}\label{scheme}
To begin exploring the FVM for solving Eq.(\ref{maineq}), we start by dividing the spatial domain into small grid volumes.
  For practical reasons, however, we prescribe the particle volumes in a specific domain, lets say $]0, R]$ where  $R$ is a finite quantity. So, the following equation is the truncated form of the CBE
	\begin{align}\label{trunceq}
	\frac{\partial{c(\varsigma,\epsilon)}}{\partial \varsigma}=  \int_0^R\int_{\epsilon}^{R} K(\rho,\sigma)b(\epsilon,\rho,\sigma)c(\varsigma,\rho)c(\varsigma,\sigma)\,d\rho\,d\sigma -\int_{0}^{R}K(\epsilon,\rho)c(\varsigma,\epsilon)c(\varsigma,\rho)\,d\rho
	\end{align}
	subject to 
	\begin{align}\label{initialtrun}
	c(0,\epsilon)\ \ = \ \ c^{in}(\epsilon) \geq 0, \ \ \ \epsilon \in ]0,R].
	\end{align}
Let us divide the domain into small grid volumes denoted as $\Lambda_i^h:=]\epsilon_{i-1/2}, \epsilon_{i+1/2}]$, where $i$ ranges from $1$ to $\mathrm{I}$. Here, $\epsilon_{1/2}=0$, $\epsilon_{\mathrm{I}+1/2}= R$, and $\Delta \epsilon_i$ is the difference between $\epsilon_{i+1/2}$ and $\epsilon_{i-1/2}$. The midpoint of the volumes is called grid point as $\epsilon_i$ for all $i$. The value of $h$ is determined as the maximum $\Delta \epsilon_i$ for all $i$. The below expression is the average value of the concentration function $c_i(t)$ in the $i^{th}$ volume, which is expressed as:	
		\begin{align}\label{meandens}
							c_{i}(\varsigma)=\frac{1}{\Delta \epsilon_i}\int_{\epsilon_{i-1/2}}^{\epsilon_{i+1/2}}c(\varsigma,\epsilon)\,d\epsilon.
						\end{align}
The time parameter domain is restricted to the range [0, T], discretized into $N$ intervals with time step $\Delta \varsigma$. The interval is characterized as 		
$$  t_n=[\varsigma_n,\varsigma_{n+1}[, \,\,  n=0,1,\ldots,N-1.$$ 
In commencing our endeavor to devise a scheme applicable to non-uniform meshes, we encounter a notable advantage: the potential to accommodate a broader domain while employing fewer mesh points compared to the traditional uniform mesh. This strategic choice promises enhanced computational efficiency and resource utilization, a significant consideration in numerical modeling and simulation.
Diverging slightly from the methodology elucidated by Bourgade and Filbet in their seminal work \cite{bourgade2008convergence}, wherein they transformed the model (coagulation and linear breakage equation) into a conservative equation using Leibniz's integral rule before proceeding with the discretization using FVM, we adopt a different approach in our present study. Here we implement FVM directly from the continuous equation (\ref{maineq}). \\

The steps we take to obtain the discretized form of the CBE (\ref{trunceq}) are as follows: The semi-discrete form is obtained by integrating Eq.(\ref{trunceq}) over $i^{th}$ cell with regard to $x$ as
\begin{align}\label{semi}
\frac{dc_i}{d\varsigma}=B_{C}(i)-D_{C}(i),
\end{align}
where
\begin{align*}
B_{C}(i)=\frac{1}{\Delta \epsilon_i}\int_{\epsilon_{i-1/2}}^{\epsilon_{i+1/2}}\int_0^{\epsilon_{\mathrm{I}+1/2}}\int_{\epsilon}^{\epsilon_{\mathrm{I}+1/2}} K(\rho,\sigma)b(\epsilon,\rho,\sigma)c(\varsigma,\rho)c(\varsigma,\sigma)d\rho\,d\sigma\,d\epsilon
\end{align*}
and
\begin{align*}
D_{C}(i)= \frac{1}{\Delta \epsilon_i}\int_{\epsilon_{i-1/2}}^{\epsilon_{i+1/2}}\int_0^{\epsilon_{\mathrm{I}+1/2}} K(\epsilon,\rho)c(\varsigma,\epsilon)c(\varsigma,\rho)d\rho\,d\epsilon
\end{align*}
along with initial distribution,
\begin{align}
c_{i}(0)=c_{i}^{in}=\frac{1}{\Delta \epsilon_i}\int_{\epsilon_{i-1/2}}^{\epsilon_{i+1/2}}c_{0}(\epsilon)\,d\epsilon.
\end{align}	
After making a few simplifications, the semi-discrete equation may be obtained by applying the midpoint rule to each of the aforementioned representations as 
\begin{align}\label{semidiscrete}
\frac{dc_{i}}{d\varsigma}=&\frac{1}{\Delta \epsilon_i}\sum_{l=1}^{\mathrm{I}}\sum_{j=i}^{\mathrm{I}}K_{j,l}c_{j}(\varsigma)c_{l}(\varsigma)\Delta \epsilon_{j}\Delta \epsilon_{l}\int_{\epsilon_{i-1/2}}^{\epsilon_{i+1/2}^{j}}b(\epsilon,\epsilon_{j},\epsilon_{l})\,d\epsilon-\sum_{j=1}^{\mathrm{I}}K_{i,j}c_{i}(\varsigma)c_{j}(\varsigma)\Delta \epsilon_{j}, 
\end{align}
where the integration limit term $\epsilon_{i+1/2}^{j}$ is expressed by
\begin{equation}
\epsilon_{i+1/2}^{j} =
\begin{cases}
\epsilon_{i}, & \text{if }\,j=i \\
\epsilon_{i+1/2}, & j\neq i.							
\end{cases}
\end{equation}
 The volume conservation law ($\frac{d}{d\varsigma}\sum_{i=1}^{I}\epsilon_i c_i \Delta \epsilon_i=0$) is not satisfied by the semi-discrete formulation (\ref{semidiscrete}). Therefore, a weight function $\Lambda_{ji}$ \cite{paul2023moments} is introduced in the death term of Eq.(\ref{semidiscrete}) to  conserve the total mass of the system and  a new form of the equation  becomes
\begin{align}\label{semidiscrete1}
\frac{dc_{i}}{d\varsigma}=&\frac{1}{\Delta \epsilon_i}\sum_{l=1}^{\mathrm{I}}\sum_{j=i}^{\mathrm{I}}K_{j,l}c_{j}(\varsigma)c_{l}(\varsigma)\Delta \epsilon_{j}\Delta \epsilon_{l}\int_{\epsilon_{i-1/2}}^{\epsilon_{i+1/2}^{j}}b(\epsilon,\epsilon_{j},\epsilon_{l})\,d\epsilon-\sum_{j=1}^{\mathrm{I}}K_{i,j}c_{i}(\varsigma)c_{j}(\varsigma)\Delta \epsilon_{j}\Lambda_{ji}, 
\end{align}
where 
\begin{align}\label{Lambda}
\Lambda_{ji}=\frac{\sum_{l=1}^{i}\epsilon_l\int_{\epsilon_{l-1/2}}^{\epsilon_{l+1/2}^{i}}b(\epsilon,\epsilon_{i},\epsilon_{j})d\epsilon}{\epsilon_i}, \quad i,j \in \{1,2,\ldots, I\}.
\end{align}
Applying explicit Euler discretization to time variable $\varsigma$  results in a completely discrete system
\begin{align}\label{fully}
c_{i}^{n+1}-c_{i}^{n}=\frac{\Delta \varsigma}{\Delta \epsilon_{i}}\sum_{l=1}^{\mathrm{I}}\sum_{j=i}^{\mathrm{I}}K_{j,l}c_{j}^{n}c_{l}^{n}\Delta \epsilon_{j}\Delta \epsilon_{l}\int_{\epsilon_{i-1/2}}^{\epsilon_{i+1/2}^{j}}b(\epsilon,\epsilon_{j},\epsilon_{l})\,d\epsilon 
-\Delta \varsigma \sum_{j=1}^{\mathrm{I}}K_{i,j}c_{i}^{n}c_{j}^{n}\Delta \epsilon_{j}\Lambda_{ji}.
\end{align}
Take into consideration a function $c^h$ on $[0, T]\times ]0,R]$ for the convergence analysis. It is represented by
\begin{align}\label{chap2:function_ch}
c^h(\varsigma,\epsilon)=\sum_{n=0}^{N-1}\sum_{i=1}^{\mathrm{I}}c_i^n\,\chi_{\Lambda_i^h}(\epsilon)\,\chi_{t_n}(\varsigma),
\end{align}
having the characteristic function as $\chi_D(\epsilon)=1$ if $\epsilon\in D$ or $0$ everywhere else. The aforementioned approximation $c^h$ implies that the constant representatives are concentrated within the cell. Also noting that $$c^h(0,\cdot)=\sum_{i=1}^{\mathrm{I}}c_i^{in} \chi_{\Lambda_i^h}(\cdot) \rightarrow c^{in} \in L^1 ]0,R[\,\, \text{as} \,\, h\rightarrow 0.$$ 
The collision kernel $K$ and breakage kernel $b$  are discretized using the finite volume approximation as $K^h$ and $b^h$  on the each cell, i.e., for all $\epsilon,\rho,\sigma \in ]0,R]$,
\begin{align}\label{chap2:function_aggregatediscrete}
	K^h(\epsilon,\rho)= \sum_{i=1}^{\mathrm{I}} \sum_{j=1}^{\mathrm{I}} K_{i,j} \chi_{\Lambda_i^h}(\epsilon) \chi_{\Lambda_j^h}(\rho),
\end{align}
and
\begin{align}\label{chap2:function_brkdiscrete}
	b^h(\epsilon,\rho,\sigma)= \sum_{i=1}^{\mathrm{I}} \sum_{j=1}^{\mathrm{I}}\sum_{l=1}^{\mathrm{I}} b_{i,j,l} \chi_{\Lambda_i^h}(\epsilon) \chi_{\Lambda_j^h}(\rho)\chi_{\Lambda_l^h}(\sigma),
	\end{align}

where $$K_{i,j}= \frac{1}{\Delta \epsilon_i \Delta \epsilon_j} \int_{\Lambda_j^h} \int_{\Lambda_i^h} K(\epsilon,\rho)d\epsilon\,d\rho, \quad  b_{i,j,l}= \frac{1}{\Delta \epsilon_i \Delta \epsilon_j \Delta \epsilon_l}\int_{\Lambda_l^h} \int_{\Lambda_j^h} \int_{\Lambda_i^h} b(\epsilon,\rho,\sigma)d\epsilon\,d\rho\,d\sigma.$$
The above formulation guarantees that $K^h \rightarrow K \in L^1$ and $b^h \rightarrow b \in L^1 $ as $h\rightarrow 0$, see \cite{bourgade2008convergence}.

\section{Estimation of Weak Solution}\label{convergence}
Our aim here is to investigate the convergence behavior of the solution $c^h$ towards a function $c$
 as the parameters $h$ and $\Delta \varsigma$
 tend towards zero.

\begin{thm}\label{maintheorem}
Assume that the kernel hypothesis $(H1)-(H2)$ holds and  $c^{in}\in \mathbb{S}^+$. Moreover, if the stability condition 
		\begin{align}\label{22}
		C(R, T)\Delta \varsigma\le \Theta< 1,
		\end{align}
	for a constant $\Theta> 0$ and time step $\Delta \varsigma$	holds for 
		\begin{align}{\label{23}}
	C(R, T):= 2\alpha N\|c^{in}\|_{\mathbb{S}}\,e^{2\alpha N M_{1}^{in} T} (R + M_{1}^{in}),
		\end{align}
	then there is a sub-sequence as $$c^h\rightarrow c\ \
		\text{in}\ \ L^\infty([0,T];L^1\,]0,R[),$$
	where c is the weak solution of problem (\ref{maineq}-\ref{initial}). This means that the function $c\geq 0$ fulfils the given condition
		\begin{align}\label{convergence0}
	\int_0^T &\int_0^{R} c(\varsigma,\epsilon)\frac{\partial\varphi}{\partial
	\varsigma}(\varsigma,\epsilon)d\epsilon\,d\varsigma -\int_0^T\int_0^{R} \int_0^{R} \int_{\epsilon}^{R}\varphi(\varsigma,\epsilon)K(\rho,\sigma)b(\epsilon,\rho,\sigma)c(\varsigma,\rho)c(\varsigma,\sigma)d\rho\,d\sigma\,d\epsilon\,d\varsigma\nonumber\\
	&+\int_0^{R} c^{in}(\epsilon)\varphi(0,\epsilon)d\epsilon+ \int_0^T \int_0^{R} \int_{0}^{R}\varphi(\varsigma,\epsilon)K(\epsilon,\rho)c(\varsigma,\epsilon)c(\varsigma,\rho)d\rho\,d\epsilon\,d\varsigma
		=0,
						\end{align}
						where $\varphi$ is compactly supported smooth functions on $[0,T]\times ]0,{R}].$
					\end{thm}
The essential goal here, as the previous theorem makes clear, is to show that when $h$ and $\Delta \varsigma$ approach zero, the family of functions $(c^h)$ weakly converges to a function $c$ in $\mathbb{S}^+$. The following theorem known as the Dunford-Pettis theorem, that provides a necessary and sufficient condition for weak convergence of $c^h$.

	\begin{thm}\label{maintheorem1} Consider a real valued sequence $c^h \in \mathbb{S}^+ $, define on $\Omega$ with $|\Omega|$ is finite.
	Assume that the sequence $\{c^h\}$ meets the following conditions
	\begin{itemize}
	\item $\{c^h\}$ is equibounded in $\mathbb{S}^+$, i.e.\
	\begin{align}\label{equiboundedness}
	\sup \|c^h\|_{\mathbb{S}^+}< \infty
	\end{align}
	\item $\{c^h\}$ is equiintegrable, iff let a increasing function $\Psi:\mathbb{R}^+\mapsto \mathbb{R}^+$ with property $$	\lim_{z\rightarrow \infty}\frac{\Psi(z)}{z}\rightarrow	\infty,$$
	hold the following condition
	\begin{align}\label{equiintegrable}
	\int_\Omega \Big(\epsilon^r+\frac{1}{\epsilon^{2p}}\Big)\Psi(|c^h|)d\epsilon < \infty.
	\end{align}
			\end{itemize}
	Then $c^h \in \mathbb{S}^+$, it implies that there exists a  subsequence of $c^h$ that converges weakly in compact set $\mathbb{S}^+$, see \cite{laurenccot2002continuous}.
	\end{thm}
Therefore, Theorem \ref{maintheorem} may be established by showing that the family $c^h$ in $\mathbb{S}^+$ is equibounded and equiintegrable, as stated in (\ref{equiboundedness}) and (\ref{equiintegrable}), respectively. The non-negativity and bound on the total number of particles for $c^h$ are discussed in the following proposition. We use the method proposed by Bourgade and Filbet \cite{bourgade2008convergence} for the proof.
	
\begin{propos}
{Let the stability criteria (\ref{22}) and kernel conditions(\ref{Collisional func}-\ref{breakage funcn}) are assumed to hold, then non-negative function $c^h$ fulfils the estimation given below for $\varsigma \in [0,T]$
\begin{align}\label{36}
\int_0^R c^h(\varsigma,\epsilon)d\epsilon \leq   \|c^{in}\|_{\mathbb{S}}\,e^{2 \alpha N M_{1}^{in} \varsigma}= \mathbb{P}(T).
\end{align}
}
\end{propos}
\begin{proof}
Mathematical induction serves as the method to establish the non-negativity of the function $c^h$. At $\varsigma = 0$, we already know that $c^h(0) \geq 0$ and is contained within $\mathbb{S}^+$. By assuming that the  functions $c^h(\varsigma^n)\geq 0 $ and
 	\begin{align}\label{mon1}
 	\int_0^{R} c^h(\varsigma^n,\epsilon)d\epsilon \le  \|c^{in}\|_{\mathbb{S}}\,e^{2 \alpha N M_{1}^{in} \varsigma^n}.
 	\end{align}
 	Our primary objective is to prove that the value of $c^h(\varsigma^{n+1})\geq 0$ . Let's examine the cell located at the boundary with an index of $i=1$. Consequently, in this particular scenario, we derive from the Eq.(\ref{fully}) as
 	\begin{align}\label{non}
 	c_{1}^{n+1}=& c_{1}^{n}+\frac{\Delta \varsigma}{\Delta \epsilon_{1}}\sum_{l=1}^{\mathrm{I}}\sum_{j=1}^{\mathrm{I}}K_{j,l}c_{j}^{n}c_{l}^{n}\Delta \epsilon_{j}\Delta \epsilon_{l}\int_{\epsilon_{1/2}}^{\epsilon_{3/2}^{j}}b(\epsilon,\epsilon_{j},\epsilon_{l})\,d\epsilon  
 	-\Delta \varsigma \sum_{j=1}^{\mathrm{I}}K_{1,j}c_{i}^{n}c_{j}^{n}\Delta \epsilon_{j}\Lambda_{j1}
 	 \nonumber \\
 \geq 	& c_{1}^{n}-\Delta \varsigma \sum_{j=1}^{\mathrm{I}}K_{1,j}c_{1}^{n}c_{j}^{n}\Delta \epsilon_{j}\Lambda_{j1}.
 	\end{align}
	Moving further, we choose collisional kernel,
	 $ K(\epsilon,\rho)= \alpha(\epsilon^{\zeta}\rho^{\eta}+\epsilon^{\eta}\rho^{\zeta}),\, \text{when}  \,\, (\epsilon,\rho)\in ]0,R[\times ]0,R[,$ and  the value of $\Lambda_{j1}$ from Eq.(\ref{Lambda}). Thus, 
	 \begin{align*}
	 c_{1}^{n+1}\geq  c_{1}^{n}-\Delta \varsigma \sum_{j=1}^{\mathrm{I}}\alpha(\epsilon_{1}^{\zeta}\epsilon_{j}^{\eta}+\epsilon_{1}^{\eta}\epsilon_{j}^{\zeta})c_{1}^{n}c_{j}^{n}\Delta \epsilon_{j} \frac{\sum_{l=1}^{1}\epsilon_l\int_{\epsilon_{l-1/2}}^{\epsilon_{l+1/2}^{1}}b(\epsilon,\epsilon_{1},\epsilon_{j})d\epsilon}{\epsilon_1},
	 \end{align*}
	 By using Young's inequality for $\alpha(\epsilon_{i}^{\zeta}\epsilon_{j}^{\eta}+\epsilon_{i}^{\eta}\epsilon_{j}^{\zeta})\leq \alpha(\epsilon_{i}+\epsilon_{j})$, we can simplify the given inequality to 
	 \begin{align}\label{non1}
	 c_{1}^{n+1}&\geq  \,c_{1}^{n}-\alpha\Delta \varsigma \sum_{j=1}^{\mathrm{I}}(\epsilon_{1}+\epsilon_{j})c_{1}^{n}c_{j}^{n}\Delta \epsilon_{j} \int_{0}^{\epsilon_1}b(\epsilon,\epsilon_{1},\epsilon_{j})d\epsilon \nonumber \\
	 &\geq \big[1-\alpha N\Delta \varsigma\big(\epsilon_1\sum_{j=1}^{\mathrm{I}}c_{j}^{n}\Delta \epsilon_{j}+\sum_{j=1}^{\mathrm{I}}\epsilon_j c_{j}^{n}\Delta \epsilon_{j} \big)\big]c_{1}^{n}.
	  \end{align} 	
	  The scheme (\ref{fully}), we considered is mass conserving, i.e. $M_1^{in}$ is a finite quantity, then
	   \begin{align}\label{non12}
	  	 c_{1}^{n+1}&\geq  
	  	  \big[1-\alpha N\Delta \varsigma\big(R\sum_{j=1}^{\mathrm{I}}c_{j}^{n}\Delta \epsilon_{j}+M_{1}^{in}\big)\big]c_{1}^{n}.
	  	  \end{align} 
	  The stability conditon (\ref{22}) on the time step and (\ref{mon1}) assist to get the nonnegativity of $c_{1}^{n+1}.$ Now, the computations for the nonnegativity of $c_{i}^{n+1}$ when $i\geq 2$ are
	  \begin{align*}
	  c_{i}^{n+1}=&c_{i}^{n}+\frac{\Delta \varsigma}{\Delta \epsilon_{i}}\sum_{l=1}^{\mathrm{I}}\sum_{j=i}^{\mathrm{I}}K_{j,l}c_{j}^{n}c_{l}^{n}\Delta \epsilon_{j}\Delta \epsilon_{l}\int_{\epsilon_{i-1/2}}^{\epsilon_{i+1/2}^{j}}b(\epsilon,\epsilon_{j},\epsilon_{l})\,d\epsilon 
	  -\Delta \varsigma \sum_{j=1}^{\mathrm{I}}K_{i,j}c_{i}^{n}c_{j}^{n}\Delta \epsilon_{j}\Lambda_{ji} \nonumber \\
	   \geq & c_{i}^{n}-\Delta \varsigma \sum_{j=1}^{\mathrm{I}}K_{i,j}c_{i}^{n}c_{j}^{n}\Delta\epsilon_{j}\frac{\sum_{l=1}^{i}\epsilon_l\int_{\epsilon_{l-1/2}}^{\epsilon_{l+1/2}^{i}}b(\epsilon,\epsilon_{i},\epsilon_{j})d\epsilon}{\epsilon_i}.
	  \end{align*}
	The condition on $K$ with Young's inequality and  $\frac{\epsilon_l}{\epsilon_i} \leq 1$ are used to obtain the simpler of the above inequality as
	\begin{align}
	c_{i}^{n+1}\geq &c_{i}^{n}-\Delta \varsigma \alpha \sum_{j=1}^{\mathrm{I}}(\epsilon_{i}+\epsilon_{j})c_{i}^{n}c_{j}^{n}\Delta \epsilon_{j} \sum_{l=1}^{i}\int_{\epsilon_{l-1/2}}^{\epsilon_{l+1/2}^{i}}b(\epsilon,\epsilon_{i},\epsilon_{j})d\epsilon \nonumber \\
	\geq &\big[1-\alpha N\Delta \varsigma\big(\epsilon_i\sum_{j=1}^{\mathrm{I}}c_{j}^{n}\Delta \epsilon_{j}+\sum_{j=1}^{\mathrm{I}}\epsilon_j c_{j}^{n}\Delta \epsilon_{j} \big)\big]c_{i}^{n} \nonumber \\
	\geq  &
		  	  \big[1-\alpha N\Delta \varsigma\big(R\sum_{j=1}^{\mathrm{I}}c_{j}^{n}\Delta \epsilon_{j}+M_{1}^{in}\big)\big]c_{i}^{n}.
	\end{align}

	  Applying again the stability condition (\ref{22}) and the $L^1$ bound (\ref{mon1}) on $c^h$   yield $c^h({t^{n+1}})\geq 0.$\\

Subsequently, it is demonstrated below that $c^h(t^{n+1})$ exhibits a similar behaviour as described in Eq.(\ref{mon1}). In order to observe this, multiplying Eq.(\ref{fully}) by $\Delta \epsilon_i$, omitting the negative component and  using summation with regard to $i$ determine the result
\begin{align}\label{equi1}
\sum_{i=1}^{\mathrm{I}}c_{i}^{n+1}\Delta \epsilon_{i} &\leq  \sum_{i=1}^{\mathrm{I}}c_{i}^{n}\Delta \epsilon_{i}+\Delta \varsigma\sum_{i=1}^{\mathrm{I}}\sum_{l=1}^{\mathrm{I}}\sum_{j=i}^{\mathrm{I}}K_{j,l}c_{j}^{n}c_{l}^{n}\Delta \epsilon_{j}\Delta \epsilon_{l}\int_{\epsilon_{i-1/2}}^{\epsilon_{i+1/2}^{j}}b(\epsilon,\epsilon_{j},\epsilon_{l})\,d\epsilon.
\end{align}
Substitute the condition of $K$ and apply the Young's inequality to simplify the above equation as
\begin{align*}
\sum_{i=1}^{\mathrm{I}}c_{i}^{n+1}\Delta \epsilon_{i} &\leq  \sum_{i=1}^{\mathrm{I}}c_{i}^{n}\Delta \epsilon_{i}+\alpha\Delta \varsigma  \sum_{i=1}^{\mathrm{I}}\sum_{l=1}^{\mathrm{I}}\sum_{j=i}^{\mathrm{I}}(\epsilon_j+\epsilon_l)c_{j}^{n}c_{l}^{n}\Delta \epsilon_{j}\Delta \epsilon_{l}\int_{\epsilon_{i-1/2}}^{\epsilon_{i+1/2}^{j}}b(\epsilon,\epsilon_{j},\epsilon_{l})\,d\epsilon.
\end{align*}
The below simplification is executed by changing  the order of summation in the second term, initial mass property, and Eq.(\ref{breakageparticle})
 \begin{align*}
 \sum_{i=1}^{\mathrm{I}}c_{i}^{n+1}\Delta \epsilon_{i} &\leq  \sum_{i=1}^{\mathrm{I}}c_{i}^{n}\Delta \epsilon_{i}+\alpha\Delta \varsigma  \sum_{l=1}^{\mathrm{I}}\sum_{j=1}^{\mathrm{I}}(\epsilon_j+\epsilon_l)c_{j}^{n}c_{l}^{n}\Delta \epsilon_{j}\Delta \epsilon_{l}\sum_{i=1}^{j}\int_{\epsilon_{i-1/2}}^{\epsilon_{i+1/2}^{j}}b(\epsilon,\epsilon_{j},\epsilon_{l})\,d\epsilon \nonumber\\
 & \leq   \sum_{i=1}^{\mathrm{I}}c_{i}^{n}\Delta \epsilon_{i}+ 2\alpha\Delta \varsigma M_{1}^{in} \sum_{i=1}^{\mathrm{I}}c_{i}^{n}\Delta \epsilon_{i}\int_{0}^{\epsilon_{j}}b(\epsilon,\epsilon_{j},\epsilon_{l})\,d\epsilon \nonumber\\ 
 & \leq \Big(1+ 2\alpha\Delta \varsigma M_{1}^{in} N\Big) \sum_{i=1}^{\mathrm{I}}c_{i}^{n}\Delta \epsilon_{i}.
 \end{align*}
 With equation (\ref{mon1}) indicating the $L^1$ bound of $c^h$ at time step $n$ and the inequality 
$1+\epsilon \leq \exp(\epsilon)$ holding for all 
$\epsilon\geq 0$, it follows that
\begin{align*}\sum_{i=1}^{\mathrm{I}}c_{i}^{n+1}\Delta \epsilon_{i} &\leq 
e^{2\alpha N M_{1}^{in} \Delta \varsigma } \|c^{in}\|_{\mathbb{S}}e^{2\alpha N M_{1}^{in}  \varsigma^{n} } \\
& = \|c^{in}\|_{\mathbb{S}}e^{2\alpha N M_{1}^{in}  \varsigma^{n+1} }.
\end{align*}
This leads to the achievement of result (\ref{36}) for $c^h(\varsigma^{n+1})$.\\
\end{proof}
The equiboundedness of $c^h$
 will be shown into the following proposition.
 \begin{propos}
 Assume stability condition (\ref{22}),  kernel growth (\ref{Collisional func}-\ref{breakage funcn}) and bound of total particle (\ref{36}) hold, then function $c^h$ for $\varsigma \in [0,T]$ yields the following condition
 \begin{align}\label{equiboundedmain}
 \int_0^R \Big(\epsilon^r+\frac{1}{\epsilon^{2p}}\Big)c^h(\varsigma,\epsilon)d\epsilon \leq   \|c^{in}\|_{\mathbb{S}}e^{2 R\alpha \lambda_{max}  \mathbb{P}(T) \varsigma}=\mathbb{P}^{*}(T), \quad \lambda_{max}=\max\{N, Q\}.
 \end{align}
 \end{propos}
 \begin{proof} 
 At $\varsigma = 0$, we already know that $c^h(0) \geq 0$ and is contained within $\mathbb{S}^+$. Assume that the  functions $c^h(\varsigma^n)$ holds
  	\begin{align}\label{equibounded0}
  	\int_0^{R} \Big(\epsilon^r+\frac{1}{\epsilon^{2p}}\Big)c^h(\varsigma^n,\epsilon)d\epsilon \le  \|c^{in}\|_{\mathbb{S}}e^{2 R\alpha \lambda_{max}  \mathbb{P}(T) \varsigma^n}.
  	\end{align}
 Now, following Eq.(\ref{fully}) by multiplying with $\Big(\epsilon_{i}^r+\frac{1}{\epsilon_{i}^{2p}}\Big)\Delta \epsilon_i$ and eliminating the negative term, the result can be determined by summing over $i$ as
\begin{align}\label{equibounded1}
\sum_{i=1}^{\mathrm{I}}\Big(\epsilon_{i}^r+\frac{1}{\epsilon_{i}^{2p}}\Big)c_{i}^{n+1}\Delta \epsilon_{i} \leq &  \underbrace{\sum_{i=1}^{\mathrm{I}}\Big(\epsilon_{i}^r+\frac{1}{\epsilon_{i}^{2p}}\Big)c_{i}^{n}\Delta \epsilon_{i}}_{I_1} \nonumber \\
+&\underbrace{\Delta \varsigma\sum_{i=1}^{\mathrm{I}}\sum_{l=1}^{\mathrm{I}}\sum_{j=i}^{\mathrm{I}}\Big(\epsilon_{i}^r+\frac{1}{\epsilon_{i}^{2p}}\Big)K_{j,l}c_{j}^{n}c_{l}^{n}\Delta \epsilon_{j}\Delta \epsilon_{l}\int_{\epsilon_{i-1/2}}^{\epsilon_{i+1/2}^{j}}b(\epsilon,\epsilon_{j},\epsilon_{l})\,d\epsilon}_{I_2+I_3}.
\end{align}
Let us deal first with the term $I_2$ of Eq.(\ref{equibounded1}). The term $\epsilon_i \leq \epsilon_j$ for $j\geq i$ and value of $K$ from (\ref{Collisional func}) lead to
\begin{align}\label{equibounded2}
I_2=\Delta \varsigma\sum_{i=1}^{\mathrm{I}}\sum_{l=1}^{\mathrm{I}}\sum_{j=i}^{\mathrm{I}}\epsilon_{i}^rK_{j,l}c_{j}^{n}c_{l}^{n}\Delta \epsilon_{j}\Delta \epsilon_{l}\int_{\epsilon_{i-1/2}}^{\epsilon_{i+1/2}^{j}}b(\epsilon,\epsilon_{j},\epsilon_{l})\,d\epsilon \nonumber \\ 
\leq \alpha \Delta \varsigma   \sum_{i=1}^{\mathrm{I}}\sum_{l=1}^{\mathrm{I}}\sum_{j=i}^{\mathrm{I}}\epsilon_{j}^r (\epsilon_j+\epsilon_l)& c_{j}^{n}c_{l}^{n}\Delta \epsilon_{j}\Delta \epsilon_{l}\int_{\epsilon_{i-1/2}}^{\epsilon_{i+1/2}^{j}}b(\epsilon,\epsilon_{j},\epsilon_{l})\,d\epsilon \nonumber \\
= \alpha \Delta \varsigma  \sum_{l=1}^{\mathrm{I}}\sum_{j=1}^{\mathrm{I}}\epsilon_{j}^r (\epsilon_j+\epsilon_l)  c_{j}^{n}& c_{l}^{n}\Delta \epsilon_{j}\Delta \epsilon_{l}\sum_{i=1}^{j}\int_{\epsilon_{i-1/2}}^{\epsilon_{i+1/2}^{j}}b(\epsilon,\epsilon_{j},\epsilon_{l})\,d\epsilon \nonumber \\
\leq  \alpha \Delta \varsigma  N  \sum_{l=1}^{\mathrm{I}}\sum_{j=1}^{\mathrm{I}}\epsilon_{j}^r (\epsilon_j+\epsilon_l)& c_{j}^{n}c_{l}^{n}\Delta \epsilon_{j}\Delta \epsilon_{l}.
\end{align}
The remaining term $I_3$ of Eq.(\ref{equibounded1}) will be simplified by using the hypothesis on $K, b$ (\ref{Collisional func}-\ref{breakage funcn}) and change of order as
\begin{align}\label{equibounded3}
I_3= \Delta \varsigma\sum_{i=1}^{\mathrm{I}}\sum_{l=1}^{\mathrm{I}}\sum_{j=i}^{\mathrm{I}}\epsilon_{i}^{-2p}K_{j,l}c_{j}^{n}c_{l}^{n}\Delta \epsilon_{j}\Delta \epsilon_{l}\int_{\epsilon_{i-1/2}}^{\epsilon_{i+1/2}^{j}}b(\epsilon,\epsilon_{j},\epsilon_{l})\,d\epsilon \nonumber \\ 
 \leq \alpha \Delta \varsigma\sum_{l=1}^{\mathrm{I}}\sum_{j=1}^{\mathrm{I}}(\epsilon_j+\epsilon_l)c_{j}^{n}c_{l}^{n} \Delta \epsilon_{j}& \Delta \epsilon_{l}\sum_{i=1}^{j}\epsilon_{i}^{-2p}\int_{\epsilon_{i-1/2}}^{\epsilon_{i+1/2}}b(\epsilon,\epsilon_{j},\epsilon_{l})\,d\epsilon \nonumber \\
=\alpha \Delta \varsigma\sum_{l=1}^{\mathrm{I}}\sum_{j=1}^{\mathrm{I}}(\epsilon_j+\epsilon_l)c_{j}^{n}c_{l}^{n}\Delta \epsilon_{j} \Delta \epsilon_{l}&\sum_{i=1}^{j}\epsilon_{i}^{-2p}b(\epsilon_{i},\epsilon_{j},\epsilon_{l})\,\Delta\epsilon_i \nonumber \\
 \leq \alpha \Delta \varsigma Q \sum_{l=1}^{\mathrm{I}}\sum_{j=1}^{\mathrm{I}}\epsilon_{j}^{-2p}(\epsilon_j+\epsilon_l) c_{j}^{n} &c_{l}^{n}\Delta \epsilon_{j}\Delta \epsilon_{l}.
\end{align}
 Eqs.(\ref{equibounded2}-\ref{equibounded3}) and quantity  $\lambda_{max}=\max\{N, Q\}$ help to get the following simplified form of equation Eq.(\ref{equibounded1}), for all $\epsilon_j,\epsilon_l \in (0,R]$
 \begin{align}
\sum_{i=1}^{\mathrm{I}}\Big(\epsilon_{i}^r+\frac{1}{\epsilon_{i}^{2p}}\Big)c_{i}^{n+1}\Delta \epsilon_{i} \leq &  \sum_{i=1}^{\mathrm{I}}\Big(\epsilon_{i}^r+\frac{1}{\epsilon_{i}^{2p}}\Big)c_{i}^{n}\Delta \epsilon_{i} \nonumber \\
 &+ \alpha \lambda_{max}\Delta \varsigma\sum_{l=1}^{\mathrm{I}}\sum_{j=1}^{\mathrm{I}}\Big(\epsilon_{j}^r+\frac{1}{\epsilon_{j}^{2p}}\Big)(\epsilon_{j}+\epsilon_{l})c_{j}^{n}c_{l}^{n}\Delta \epsilon_{j}\Delta \epsilon_{l}\nonumber \\
  \leq & \Big(1+2 R\alpha \lambda_{max}\Delta \varsigma \sum_{l=1}^{\mathrm{I}}c_{l}^{n}\Delta \epsilon_{l}\Big) \sum_{i=1}^{\mathrm{I}}\Big(\epsilon_{i}^r+\frac{1}{\epsilon_{i}^{2p}}\Big)c_{i}^{n}\Delta \epsilon_{i}.
 \end{align}
 The assistance of Eq.(\ref{36}), bound of $c^h$  for time step  $n$ in Eq.(\ref{equibounded0}) and $1+\epsilon \leq \exp(\epsilon)$ yield the bound of $c^h(\varsigma^{n+1})$ in $\mathbb{S}^{+}$
 \begin{align*}
 \sum_{i=1}^{\mathrm{I}}\Big(\epsilon_{i}^r+\frac{1}{\epsilon_{i}^{2p}}\Big)c_{i}^{n+1}\Delta \epsilon_{i}\leq& e^{2 R\alpha \lambda_{max}\Delta \varsigma \mathbb{P}(T)}\sum_{i=1}^{\mathrm{I}}\Big(\epsilon_{i}^r+\frac{1}{\epsilon_{i}^{2p}}\Big)c_{i}^{n}\Delta \epsilon_{i}\\
 \leq & \|c^{in}\|_{\mathbb{S}}e^{2 R\alpha \lambda_{max}  \mathbb{P}(T) \varsigma^{n+1}}.
 \end{align*}
 The equiboundedness of $c^h$ (\ref{equiboundedmain}) is demonstrated in $\mathbb{S}^{+}$.
 \end{proof}

Let us define a particular class of convex functions denoted as 
 $C_{V P,\infty}$ for proving the uniform integrability. Assume  non-negative $ \Psi \in C^{\infty}([0,\infty[)$ be a convex function belonging to ${C}_{VP, \infty}$ with the following characteristics:
	\begin{description}
	\item[(1)] $\Psi(0)=0,\ \Psi'(0)=1$ and $\Psi'$ is concave;
	\item[(2)] $\lim_{z \to \infty} \Psi'(z) =\lim_{z \to \infty} \frac{ \Psi(z)}{z}=\infty$;
	\item[(3)] for $\theta \in ]1, 2[$,
	\begin{align}\label{Tproperty}
	\Pi_{\theta}(\Psi):= \sup_{z \ge 0} \bigg\{   \frac{ \Psi(z)}{z^{\theta}} \bigg\} < \infty.
	\end{align}
	\end{description}
	It is given that, $c^{in}\in \mathbb{S}^{+}$. Therefore, by the refined version of De la Vall\'{e}e Poussin theorem \cite{laurenccot2014weak}, a continuously differentiable function
	$\mathrm{\Psi}\geq 0$ defined  on
	$\mathbb{R}^{+}$ and follow the properties of characteristics {\bf{(1)}} such that
	$$\frac{\mathrm{\Psi}(z)}{z} \rightarrow \infty,\ \ \text{as}\ \
	z \rightarrow \infty$$ and
	\begin{align}\label{convex}
	\mathbb{I}:=\int_0^{R}\Big(\epsilon^r+\frac{1}{\epsilon^{2p}}\Big) \mathrm{\Psi}(c^{in})(\epsilon)d\epsilon< \infty.
	\end{align}
	\begin{lem} [\cite{laurenccot2002continuous}, Lemma B.1.]\label{lemma}
	Let $\mathrm{\Psi}\in {C}_{VP, \infty}$. 
Then $\forall$  $(\epsilon,\rho)\in \mathbb{R}^{+}\times \mathbb{R}^{+},$
	$$\epsilon\mathrm{\Psi}'(\rho)\leq \mathrm{\Psi}(\epsilon)+\mathrm{\Psi}(\rho).$$
	\end{lem}
	
The concept of equiintegrability is presently being analysed in the following statement.
	\begin{propos}\label{equiintegrability}
Assume 
$c^{in}\geq 0\in \mathbb{S}^+$, and equation (\ref{fully}) produces the family 
$(c^h)$ for any  $h$ and $\Delta \varsigma$, with  $\Delta \varsigma$ satisfying the condition (\ref{22}). Then, the sequence 
$(c^h)$ is weakly sequentially compact in 
$\mathbb{S}$.
	\end{propos}
\begin{proof}
	Here, we enjoy obtaining a result for the function family $c^h$ that is similar to (\ref{convex}). The integral of $(c^h)$ may be written as follows using the sequence $c^{n}_{i}$
		\begin{align*}
		\int_0^T \int_0^{R}\Big(\epsilon^r+\frac{1}{\epsilon^{2p}}\Big) \mathrm{\Psi}(c^h(\varsigma,\epsilon))d\epsilon\,d\varsigma=&\sum_{n=0}^{N-1}\sum_{i=1}^{\mathrm{I}}\int_{t_n}\int_{\Lambda_i^h}\Big(\epsilon_{i}^r+\frac{1}{\epsilon_{i}^{2p}}\Big)\mathrm{\Psi}
		\bigg(\sum_{k=0}^{N-1}\sum_{j=1}^{\mathrm{I}}c_j^k\chi_{\Lambda_j^h}(\epsilon)\chi_{t_k}(\varsigma)\bigg)d\epsilon\,d\varsigma\\
		=&\sum_{n=0}^{N-1}\sum_{i=1}^{\mathrm{I}}\Big(\epsilon_{i}^r+\frac{1}{\epsilon_{i}^{2p}}\Big)\mathrm{\Psi}(c_i^n)\Delta \epsilon_i \Delta \varsigma.
		\end{align*}
The convex property of $\mathrm{\Psi}$ and ${\mathrm{\Psi}}^{'}\geq 0$, together with the discrete Eq.(\ref{fully}), all imply that
	\begin{align}\label{Equi1}
\sum_{i=1}^{\mathrm{I}} \Big(\epsilon_{i}^r+\frac{1}{\epsilon_{i}^{2p}}\Big)[\mathrm{\Psi}(c_i^{n+1})-\mathrm{\Psi}(c_i^{n})]\Delta \epsilon_{i}& \leq \sum_{i=1}^{\mathrm{I}}\Big(\epsilon_{i}^r+\frac{1}{\epsilon_{i}^{2p}}\Big) \left(c_i^{n+1}-c_i^{n}\right)\mathrm{\Psi}^{'}(c_i^{n+1})\Delta \epsilon_{i} 
 =T_1-T_2.
	\end{align}
Expand the first term ($T_1$) of the above equation over $j$ for particular $i$ 
\begin{align}\label{Equiint2}
T_1 =&\Delta \varsigma \sum_{i=1}^{\mathrm{I}}\sum_{l=1}^{\mathrm{I}}\sum_{j=i}^{\mathrm{I}}\Big(\epsilon_{i}^r+\frac{1}{\epsilon_{i}^{2p}}\Big) K_{j,l}c_{j}^{n}c_{l}^{n}\mathrm{\Psi}^{'}(c_i^{n+1})\Delta \epsilon_{j}\Delta \epsilon_{l}\int_{\epsilon_{i-1/2}}^{\epsilon_{i+1/2}^{j}}b(\epsilon,\epsilon_{j},\epsilon_{l})\,d\epsilon \nonumber \\
=&\Delta \varsigma \sum_{i=1}^{\mathrm{I}}\sum_{l=1}^{\mathrm{I}}\Big(\epsilon_{i}^r+\frac{1}{\epsilon_{i}^{2p}}\Big) K_{i,l}c_{i}^{n}c_{l}^{n}\mathrm{\Psi}^{'}(c_i^{n+1})\Delta \epsilon_{i}\Delta \epsilon_{l}\int_{0}^{\epsilon_i}b(\epsilon,\epsilon_{i},\epsilon_{l})\,d\epsilon \nonumber \\
&+ \Delta \varsigma \sum_{i=1}^{\mathrm{I}}\sum_{l=1}^{\mathrm{I}}\sum_{j=i+1}^{\mathrm{I}}\Big(\epsilon_{i}^r+\frac{1}{\epsilon_{i}^{2p}}\Big) K_{j,l}c_{j}^{n}c_{l}^{n}\mathrm{\Psi}^{'}(c_i^{n+1})\Delta \epsilon_{j}\Delta \epsilon_{l}\int_{\epsilon_{i-1/2}}^{\epsilon_{i+1/2}^{j}}b(\epsilon,\epsilon_{j},\epsilon_{l})\,d\epsilon.
\end{align}
Similarly,  open the second term ($T_2$) over $l$ leads to
\begin{align}\label{Equiint3}
T_2=&\Delta \varsigma \sum_{i=1}^{\mathrm{I}}\sum_{j=1}^{\mathrm{I}}\Big(\epsilon_{i}^r+\frac{1}{\epsilon_{i}^{2p}}\Big)K_{i,j}c_{i}^{n}c_{j}^{n}\Delta \epsilon_{i}\Delta \epsilon_{j}\mathrm{\Psi}^{'}(c_i^{n+1})\frac{\sum_{l=1}^{i}\epsilon_l\int_{\epsilon_{l-1/2}}^{\epsilon_{l+1/2}^{i}}b(\epsilon,\epsilon_{i},\epsilon_{j})d\epsilon}{\epsilon_i}\nonumber \\
= &\Delta \varsigma \sum_{i=1}^{\mathrm{I}}\sum_{j=1}^{\mathrm{I}}\Big(\epsilon_{i}^r+\frac{1}{\epsilon_{i}^{2p}}\Big)K_{i,j}c_{i}^{n}c_{j}^{n}\Delta \epsilon_{i}\Delta \epsilon_{j}\mathrm{\Psi}^{'}(c_i^{n+1})\int_{0}^{\epsilon_{i}}b(\epsilon,\epsilon_{i},\epsilon_{j})d\epsilon \nonumber \\
&+ \Delta \varsigma \sum_{i=1}^{\mathrm{I}}\sum_{j=1}^{\mathrm{I}}\Big(\epsilon_{i}^r+\frac{1}{\epsilon_{i}^{2p}}\Big)K_{i,j}c_{i}^{n}c_{j}^{n}\Delta \epsilon_{i}\Delta \epsilon_{j}\mathrm{\Psi}^{'}(c_i^{n+1})\frac{\sum_{l=1}^{i-1}\epsilon_l\int_{\epsilon_{l-1/2}}^{\epsilon_{l+1/2}^{i}}b(\epsilon,\epsilon_{i},\epsilon_{j})d\epsilon}{\epsilon_i}.&
\end{align}
Substitute the values of Eqs.(\ref{Equiint2}-\ref{Equiint3}) and the condition on $K$ in Eq.(\ref{Equi1}) to get the new form for the analysis as
\begin{align}\label{Equiint4}
\sum_{i=1}^{\mathrm{I}} \Big(\epsilon_{i}^r+\frac{1}{\epsilon_{i}^{2p}}\Big)[\mathrm{\Psi}(c_i^{n+1})-\mathrm{\Psi}(c_i^{n})]\Delta \epsilon_{i} \nonumber \\
\leq  \Delta \varsigma \sum_{i=1}^{\mathrm{I}}\sum_{l=1}^{\mathrm{I}}\sum_{j=i+1}^{\mathrm{I}}\Big(\epsilon_{i}^r+\frac{1}{\epsilon_{i}^{2p}}\Big) K_{j,l}& c_{j}^{n}c_{l}^{n}\mathrm{\Psi}^{'}(c_i^{n+1})\Delta \epsilon_{j}\Delta \epsilon_{l}\int_{\epsilon_{i-1/2}}^{\epsilon_{i+1/2}^{j}}b(\epsilon,\epsilon_{j},\epsilon_{l})\,d\epsilon \nonumber \\
 \leq \alpha \Delta \varsigma \sum_{i=1}^{\mathrm{I}}\sum_{l=1}^{\mathrm{I}}\sum_{j=i+1}^{\mathrm{I}}\Big(\epsilon_{i}^r+\frac{1}{\epsilon_{i}^{2p}}\Big) (\epsilon_j&+\epsilon_l)c_{j}^{n}c_{l}^{n}b(\epsilon_i,\epsilon_{j},\epsilon_{l})\mathrm{\Psi}^{'}(c_i^{n+1})\Delta \epsilon_{j}\Delta \epsilon_{l}\Delta \epsilon_{i}\nonumber \\
 =:S_1+S_2.~~~~~~~~~~~~~~~~~~~~~~~~~~~~~&
\end{align}
The convexity result mentioned in Lemma \ref{lemma} helps $S_1$ to rewite into
\begin{align}\label{Equiint5}
S_1\leq &\alpha \Delta \varsigma \sum_{i=1}^{\mathrm{I}}\sum_{l=1}^{\mathrm{I}}\sum_{j=i+1}^{\mathrm{I}}\Big(\epsilon_{i}^r+\frac{1}{\epsilon_{i}^{2p}}\Big) \epsilon_jc_{j}^{n}\Delta \epsilon_{j}c_{l}^{n}\Delta \epsilon_{l}[\mathrm{\Psi}(c_i^{n+1})+\mathrm{\Psi}(b(\epsilon_i,\epsilon_{j},\epsilon_{l}))]\Delta \epsilon_{i} \nonumber \\
 = &\alpha \Delta \varsigma \sum_{i=1}^{\mathrm{I}}\sum_{l=1}^{\mathrm{I}}\sum_{j=i+1}^{\mathrm{I}}\Big(\epsilon_{i}^r+\frac{1}{\epsilon_{i}^{2p}}\Big) \epsilon_jc_{j}^{n}\Delta \epsilon_{j}c_{l}^{n}\Delta \epsilon_{l}\mathrm{\Psi}(c_i^{n+1})\Delta \epsilon_{i}\nonumber \\
 &+\alpha \Delta \varsigma \sum_{i=1}^{\mathrm{I}}\sum_{l=1}^{\mathrm{I}}\sum_{j=i+1}^{\mathrm{I}}\Big(\epsilon_{i}^r+\frac{1}{\epsilon_{i}^{2p}}\Big) \epsilon_jc_{j}^{n}\Delta \epsilon_{j}c_{l}^{n}\Delta \epsilon_{l}\mathrm{\Psi}(b(\epsilon_i,\epsilon_{j},\epsilon_{l}))\Delta \epsilon_{i}.  
\end{align}
Second term in RHS of Eq.(\ref{Equiint5}) can be simplified using the Eq.(\ref{Tproperty})  that yields
\begin{align}\label{Equiint6}
  \alpha \Delta \varsigma \sum_{i=1}^{\mathrm{I}}\sum_{l=1}^{\mathrm{I}}\sum_{j=i+1}^{\mathrm{I}}\Big(\epsilon_{i}^r+\frac{1}{\epsilon_{i}^{2p}}\Big) \epsilon_jc_{j}^{n}\Delta \epsilon_{j}c_{l}^{n}\Delta \epsilon_{l}\mathrm{\Psi}(b(\epsilon_i,\epsilon_{j},\epsilon_{l}))\Delta \epsilon_{i} \nonumber \\
  =\alpha \Delta \varsigma \sum_{i=1}^{\mathrm{I}}\sum_{l=1}^{\mathrm{I}}\sum_{j=i+1}^{\mathrm{I}}\Big(\epsilon_{i}^r+\frac{1}{\epsilon_{i}^{2p}}\Big) \epsilon_jc_{j}^{n}\Delta \epsilon_{j}c_{l}^{n}\Delta \epsilon_{l}\frac{{\mathrm{\Psi}}(b(\epsilon_i,\epsilon_{j},\epsilon_{l}))}{\{b(\epsilon_i,\epsilon_{j},\epsilon_{l})\}^{\theta}}{b(\epsilon_i,\epsilon_{j},\epsilon_{l})}^{\theta}\Delta \epsilon_{i} \nonumber \\
  \leq \alpha \Delta \varsigma \Pi_{\theta}(\Psi) \sum_{i=1}^{\mathrm{I}}\sum_{l=1}^{\mathrm{I}}\sum_{j=i+1}^{\mathrm{I}}\Big(\epsilon_{i}^r+\frac{1}{\epsilon_{i}^{2p}}\Big) \epsilon_jc_{j}^{n}\Delta \epsilon_{j}c_{l}^{n}\Delta \epsilon_{l}{b(\epsilon_i,\epsilon_{j},\epsilon_{l})}^{\theta}\Delta \epsilon_{i}\nonumber \\
  =\alpha \Delta \varsigma \Pi_{\theta}(\Psi) \sum_{l=1}^{\mathrm{I}}\sum_{j=1}^{\mathrm{I}} \epsilon_jc_{j}^{n}\Delta \epsilon_{j}c_{l}^{n}\Delta \epsilon_{l}\sum_{i=1}^{j-1}\Big(\epsilon_{i}^r+\frac{1}{\epsilon_{i}^{2p}}\Big){b(\epsilon_i,\epsilon_{j},\epsilon_{l})}^{\theta}\Delta \epsilon_{i}.
\end{align}
The RHS of Eq.(\ref{Equiint6}) is further written in two terms, where the first term is
\begin{align*}
(\ref{Equiint6})(a)= \alpha \Delta \varsigma \Pi_{\theta}(\Psi) \sum_{l=1}^{\mathrm{I}}\sum_{j=1}^{\mathrm{I}} \epsilon_jc_{j}^{n}\Delta \epsilon_{j}c_{l}^{n}\Delta \epsilon_{l}\sum_{i=1}^{j-1}\epsilon_{i}^r{b(\epsilon_i,\epsilon_{j},\epsilon_{l})}^{\theta}\Delta \epsilon_{i}\nonumber \\ \leq 
 \alpha \Delta \varsigma \Pi_{\theta}(\Psi)R^{r+1} \sum_{l=1}^{\mathrm{I}}\sum_{j=1}^{\mathrm{I}} c_{j}^{n}\Delta \epsilon_{j}c_{l}^{n}\Delta \epsilon_{l}\sum_{i=1}^{j-1}\int_{\epsilon_{i-1/2}}^{\epsilon_{i+1/2}}{b(\epsilon,\epsilon_{j},\epsilon_{l})}^{\theta}d\epsilon \nonumber\\
 \leq  \alpha \Delta \varsigma \Pi_{\theta}(\Psi)R^{r+1} \sum_{l=1}^{\mathrm{I}}\sum_{j=1}^{\mathrm{I}} c_{j}^{n}\Delta \epsilon_{j}c_{l}^{n}\Delta \epsilon_{l}\int_{0}^{\epsilon_j}{b(\epsilon,\epsilon_{j},\epsilon_{l})}^{\theta}d\epsilon.
\end{align*}
An upper bound for the aforementioned equation is provided by Eq.(\ref{breakage funcn}), as 
\begin{align}\label{Equiint7}
(\ref{Equiint6})(a)\leq \alpha \Delta \varsigma \Pi_{\theta}(\Psi)R^{r+1}Q\sum_{l=1}^{\mathrm{I}}\sum_{j=1}^{\mathrm{I}} {\epsilon_j}^{1-\tau}c_{j}^{n}\Delta \epsilon_{j}c_{l}^{n}\Delta \epsilon_{l}.
\end{align}
Now, write the second term of the RHS of Eq.(\ref{Equiint6}) to simplify using the Eq.(\ref{breakage funcn}) that leads to
\begin{align}\label{Equiint8}
(\ref{Equiint6})(b)= &\alpha \Delta \varsigma \Pi_{\theta}(\Psi) \sum_{l=1}^{\mathrm{I}}\sum_{j=1}^{\mathrm{I}} \epsilon_jc_{j}^{n}\Delta \epsilon_{j}c_{l}^{n}\Delta \epsilon_{l}\sum_{i=1}^{j-1}\epsilon_{i}^{-2p}{b(\epsilon_i,\epsilon_{j},\epsilon_{l})}^{\theta}\Delta \epsilon_{i} \nonumber \\
&\leq \alpha \Delta \varsigma \Pi_{\theta}(\Psi)R \sum_{l=1}^{\mathrm{I}}\sum_{j=1}^{\mathrm{I}} c_{j}^{n}\Delta \epsilon_{j}c_{l}^{n}\Delta \epsilon_{l}\int_{0}^{\epsilon_j}\epsilon^{-2p}{b(\epsilon_i,\epsilon_{j},\epsilon_{l})}^{\theta}d \epsilon \nonumber \\
&\leq \alpha \Delta \varsigma \Pi_{\theta}(\Psi)RQ \sum_{l=1}^{\mathrm{I}}\sum_{j=1}^{\mathrm{I}} {\epsilon_j}^{-2p+1-\tau}c_{j}^{n}\Delta \epsilon_{j}c_{l}^{n}\Delta \epsilon_{l}.
\end{align}
Combine the Eqs.(\ref{Equiint5}-\ref{Equiint8}), with the help of Eq.(\ref{36}) and Eq.(\ref{equiboundedmain}) get a bound
\begin{align}\label{Equiint9}
S_1 \leq &\alpha \Delta \varsigma \mathbb{P}(T)M_{1}^{in}\sum_{i=1}^{\mathrm{I}}\Big(\epsilon_{i}^r+\frac{1}{\epsilon_{i}^{2p}}\Big) \mathrm{\Psi}(c_i^{n+1})\Delta \epsilon_{i} +\alpha \Delta \varsigma \Pi_{\theta}(\Psi)R^{r+1}Q \mathbb{P}(T)\mathbb{P}^{*}(T) \nonumber \\
& +\alpha \Delta \varsigma \Pi_{\theta}(\Psi)RQ\mathbb{P}(T)\mathbb{P}^{*}(T)\nonumber \\
\leq  &\alpha \Delta \varsigma \mathbb{P}(T)M_{1}^{in}\sum_{i=1}^{\mathrm{I}}\Big(\epsilon_{i}^r+\frac{1}{\epsilon_{i}^{2p}}\Big) \mathrm{\Psi}(c_i^{n+1})\Delta \epsilon_{i}+\alpha \Delta \varsigma \Pi_{\theta}(\Psi)RQ \mathbb{P}(T)\mathbb{P}^{*}(T)[R^{r}+1].
\end{align}
The seond term $S_2$ of Eq.(\ref{Equiint4}) is estimated   using the convexity result in Lemma \ref{lemma} as
\begin{align}\label{Equiint10}
S_2\leq &\alpha \Delta \varsigma \sum_{i=1}^{\mathrm{I}}\sum_{l=1}^{\mathrm{I}}\sum_{j=i+1}^{\mathrm{I}}\Big(\epsilon_{i}^r+\frac{1}{\epsilon_{i}^{2p}}\Big) \epsilon_lc_{j}^{n}\Delta \epsilon_{j}c_{l}^{n}\Delta \epsilon_{l}[\mathrm{\Psi}(c_i^{n+1})+\mathrm{\Psi}(b(\epsilon_i,\epsilon_{j},\epsilon_{l}))]\Delta \epsilon_{i} \nonumber \\
 = &\alpha \Delta \varsigma \sum_{i=1}^{\mathrm{I}}\sum_{l=1}^{\mathrm{I}}\sum_{j=i+1}^{\mathrm{I}}\Big(\epsilon_{i}^r+\frac{1}{\epsilon_{i}^{2p}}\Big) \epsilon_l c_{l}^{n}\Delta \epsilon_{l}c_{j}^{n}\Delta \epsilon_{j}\mathrm{\Psi}(c_i^{n+1})\Delta \epsilon_{i}\nonumber \\
 &+\alpha \Delta \varsigma \sum_{i=1}^{\mathrm{I}}\sum_{l=1}^{\mathrm{I}}\sum_{j=i+1}^{\mathrm{I}}\Big(\epsilon_{i}^r+\frac{1}{\epsilon_{i}^{2p}}\Big) \epsilon_l c_{l}^{n}\Delta \epsilon_{l}c_{j}^{n}\Delta \epsilon_{j}\mathrm{\Psi}(b(\epsilon_i,\epsilon_{j},\epsilon_{l}))\Delta \epsilon_{i}. 
\end{align}
Second term in RHS of Eq.(\ref{Equiint10}) can be reduced using the Eq.(\ref{Tproperty}) as
\begin{align}\label{Equiint11}
  \alpha \Delta \varsigma \sum_{i=1}^{\mathrm{I}}\sum_{l=1}^{\mathrm{I}}\sum_{j=i+1}^{\mathrm{I}}\Big(\epsilon_{i}^r+\frac{1}{\epsilon_{i}^{2p}}\Big) \epsilon_l c_{l}^{n}\Delta \epsilon_{l}c_{j}^{n}\Delta \epsilon_{j}\mathrm{\Psi}(b(\epsilon_i,\epsilon_{j},\epsilon_{l}))\Delta \epsilon_{i} \nonumber \\
  =\alpha \Delta \varsigma \sum_{i=1}^{\mathrm{I}}\sum_{l=1}^{\mathrm{I}}\sum_{j=i+1}^{\mathrm{I}}\Big(\epsilon_{i}^r+\frac{1}{\epsilon_{i}^{2p}}\Big) \epsilon_l c_{l}^{n}\Delta \epsilon_{l} c_{j}^{n}\Delta \epsilon_{j}\frac{{\mathrm{\Psi}}(b(\epsilon_i,\epsilon_{j},\epsilon_{l}))}{\{b(\epsilon_i,\epsilon_{j},\epsilon_{l})\}^{\theta}}{b(\epsilon_i,\epsilon_{j},\epsilon_{l})}^{\theta}\Delta \epsilon_{i} \nonumber \\
  \leq \alpha \Delta \varsigma \Pi_{\theta}(\Psi) \sum_{l=1}^{\mathrm{I}}\sum_{j=1}^{\mathrm{I}} \epsilon_l c_{l}^{n}\Delta \epsilon_{l}c_{j}^{n}\Delta \epsilon_{j}\sum_{i=1}^{j-1}\Big(\epsilon_{i}^r+\frac{1}{\epsilon_{i}^{2p}}\Big){b(\epsilon_i,\epsilon_{j},\epsilon_{l})}^{\theta}\Delta \epsilon_{i}.
\end{align}
Eq.(\ref{Equiint11}) will be written in two terms, where the first term  (\ref{Equiint11})$(a)$ is expressed, and calculated using  Eq.(\ref{breakage funcn}) as below
\begin{align*}
(\ref{Equiint11})(a)= &\alpha \Delta \varsigma \Pi_{\theta}(\Psi) \sum_{l=1}^{\mathrm{I}}\sum_{j=1}^{\mathrm{I}} \epsilon_l c_{l}^{n}\Delta \epsilon_{l}c_{j}^{n}\Delta \epsilon_{j}\sum_{i=1}^{j-1}\epsilon_{i}^r{b(\epsilon_i,\epsilon_{j},\epsilon_{l})}^{\theta}\Delta \epsilon_{i}\nonumber \\ 
&\leq 
 \alpha \Delta \varsigma \Pi_{\theta}(\Psi)R^{r} \sum_{l=1}^{\mathrm{I}}\sum_{j=1}^{\mathrm{I}} \epsilon_{l}c_{l}^{n}\Delta \epsilon_{l}c_{j}^{n}\Delta \epsilon_{j}\sum_{i=1}^{j-1}\int_{\epsilon_{i-1/2}}^{\epsilon_{i+1/2}}{b(\epsilon,\epsilon_{j},\epsilon_{l})}^{\theta}d\epsilon \nonumber\\
& \leq  \alpha \Delta \varsigma \Pi_{\theta}(\Psi)R^{r} \sum_{l=1}^{\mathrm{I}}\sum_{j=1}^{\mathrm{I}} \epsilon_{l}c_{l}^{n}\Delta \epsilon_{l}c_{j}^{n}\Delta \epsilon_{j}\int_{0}^{\epsilon_j}{b(\epsilon,\epsilon_{j},\epsilon_{l})}^{\theta}d\epsilon.
\end{align*}
Thanks to Eq.(\ref{breakage funcn}), the above expression can be simplified as  
\begin{align}\label{Equiint12}
(\ref{Equiint11})(a)\leq \alpha \Delta \varsigma \Pi_{\theta}(\Psi)R^{r}Q\sum_{l=1}^{\mathrm{I}}\sum_{j=1}^{\mathrm{I}} {\epsilon_j}^{1-\tau}c_{j}^{n}\Delta \epsilon_{j}\epsilon_{l} c_{l}^{n}\Delta \epsilon_{l}. 
\end{align}
Now, evaluating the second term of Eq.(\ref{Equiint11}) with the help of  Eq.(\ref{breakage funcn}) provides
\begin{align}\label{Equiint13}
(\ref{Equiint11})(b)= &\alpha \Delta \varsigma \Pi_{\theta}(\Psi) \sum_{l=1}^{\mathrm{I}}\sum_{j=1}^{\mathrm{I}} \epsilon_l c_{l}^{n}\Delta \epsilon_{l}c_{j}^{n}\Delta \epsilon_{j}\sum_{i=1}^{j-1}\epsilon_{i}^{-2p}{b(\epsilon_i,\epsilon_{j},\epsilon_{l})}^{\theta}\Delta \epsilon_{i} \nonumber \\
&\leq \alpha \Delta \varsigma \Pi_{\theta}(\Psi) \sum_{l=1}^{\mathrm{I}}\sum_{j=1}^{\mathrm{I}} \epsilon_l c_{l}^{n}\Delta \epsilon_{l}c_{j}^{n}\Delta \epsilon_{j}\int_{0}^{\epsilon_j}\epsilon^{-2p}{b(\epsilon_i,\epsilon_{j},\epsilon_{l})}^{\theta}d \epsilon \nonumber \\
&\leq \alpha \Delta \varsigma \Pi_{\theta}(\Psi)Q \sum_{l=1}^{\mathrm{I}}\sum_{j=1}^{\mathrm{I}} {\epsilon_j}^{-2p+1-\tau}c_{j}^{n}\Delta \epsilon_{j}\epsilon_{l} c_{l}^{n}\Delta \epsilon_{l}.
\end{align}
Combine all the Eqs.(\ref{Equiint10}-\ref{Equiint13}), Eqs.(\ref{36},\ref{equiboundedmain}) and $M_{1}^{in}$ to get the bound of $S_2$
\begin{align}\label{Equiint14}
S_2\leq & \alpha \Delta \varsigma \mathbb{P}(T)M_{1}^{in}\sum_{i=1}^{\mathrm{I}}\Big(\epsilon_{i}^r+\frac{1}{\epsilon_{i}^{2p}}\Big) \mathrm{\Psi}(c_i^{n+1})\Delta \epsilon_{i} +\alpha \Delta \varsigma \Pi_{\theta}(\Psi)R^{r}Q M_{1}^{in}\mathbb{P}^{*}(T) \nonumber \\
& +\alpha \Delta \varsigma \Pi_{\theta}(\Psi)Q M_{1}^{in}\mathbb{P}^{*}(T)\nonumber \\
\leq  &\alpha \Delta \varsigma \mathbb{P}(T)M_{1}^{in}\sum_{i=1}^{\mathrm{I}}\Big(\epsilon_{i}^r+\frac{1}{\epsilon_{i}^{2p}}\Big) \mathrm{\Psi}(c_i^{n+1})\Delta \epsilon_{i}+\alpha \Delta \varsigma \Pi_{\theta}(\Psi) Q M_{1}^{in}\mathbb{P}^{*}(T)[R^{r}+1].
\end{align}
All the estimations computed in Eqs.(\ref{Equiint9},\ref{Equiint14}) used in Eq.(\ref{Equiint4}) imply
\begin{align}
\sum_{i=1}^{\mathrm{I}} \Big(\epsilon_{i}^r+\frac{1}{\epsilon_{i}^{2p}}\Big)[\mathrm{\Psi}(c_i^{n+1})-\mathrm{\Psi}(c_i^{n})]\Delta \epsilon_{i} \leq & 2\alpha \Delta \varsigma \mathbb{P}(T)M_{1}^{in}\sum_{i=1}^{\mathrm{I}}\Big(\epsilon_{i}^r+\frac{1}{\epsilon_{i}^{2p}}\Big) \mathrm{\Psi}(c_i^{n+1})\Delta \epsilon_{i}\nonumber \\
&+\alpha \Delta \varsigma \Pi_{\theta}(\Psi)Q \mathbb{P}^{*}(T)[R^{r}+1][R\mathbb{P}(T)+M_{1}^{in}].
\end{align}
 The aforementioned expression can be written as
\begin{align*}
(1-2\alpha \Delta \varsigma \mathbb{P}(T)M_{1}^{in})\sum_{i=1}^{\mathrm{I}}\Big(\epsilon_{i}^r+\frac{1}{\epsilon_{i}^{2p}}\Big)\mathrm{\Psi}(c_i^{n+1})\Delta \epsilon_{i}  \leq & \sum_{i=1}^{\mathrm{I}}\Big(\epsilon_{i}^r+\frac{1}{\epsilon_{i}^{2p}}\Big)\mathrm{\Psi}(c_i^{n})\Delta \epsilon_{i}+B,
\end{align*}
where $$B=\alpha \Delta \varsigma \Pi_{\theta}(\Psi)Q \mathbb{P}^{*}(T)[R^{r}+1][R\mathbb{P}(T)+M_{1}^{in}].$$
Using the stability condition (\ref{22}), the above inequality transforms into 
\begin{align*}
\sum_{i=1}^{\mathrm{I}}  \Big(\epsilon_{i}^r+\frac{1}{\epsilon_{i}^{2p}}\Big)\mathrm{\Psi}(c_i^{n+1})\Delta \epsilon_{i}  \le A \sum_{i=1}^{\mathrm{I}}   \Big(\epsilon_{i}^r+\frac{1}{\epsilon_{i}^{2p}}\Big)\mathrm{\Psi}(c_i^{n})\Delta \epsilon_{i} +  B^{*},
\end{align*} 
where 
$$  A= \frac{1}{(1-2\alpha \Delta \varsigma \mathbb{P}(T)M_{1}^{in})},  \,\, 
B^{*} =\frac{B}{(1-2\alpha \Delta \varsigma \mathbb{P}(T)M_{1}^{in})}.
$$
Therefore,
	\begin{align}\label{Equi4}
	\sum_{i=1}^{\mathrm{I}}   \Big(\epsilon_{i}^r+\frac{1}{\epsilon_{i}^{2p}}\Big)\mathrm{\Psi}(c_i^{n})\Delta \epsilon_{i} \le A^{n} \sum_{i=1}^{\mathrm{I}}  \Big(\epsilon_{i}^r+\frac{1}{\epsilon_{i}^{2p}}\Big)\mathrm{\Psi}(c_i^{in})\Delta \epsilon_{i}+  B^{*} \frac{A^{n} -1}{A-1}.
	\end{align}
	Thanks to Property (\ref{convex}), we obtain
		\begin{align}\label{Equi5}
	\int_0^{R} \Big(\epsilon^r+\frac{1}{\epsilon^{2p}}\Big) \mathrm{\Psi}(c^{h}(\varsigma, \epsilon))\,d\epsilon \le 
	 A^{n}\mathbb{I}+  B^{*} \frac{A^{n} -1}{A-1} < \infty, \quad \text{for all}\ \ \ \varsigma \in [0,T].
	\end{align}
By applying Theorem \ref{maintheorem1}, we conclude that the sequence $(c^h)$  is weakly compact in  
$\mathbb{S}^+$. This guarantees the existence of a subsequence that converges weakly to 
 $c \in \mathbb{S}^{+}$ as $h \rightarrow 0$.
	\end{proof}	
	Next, we need to show that sequence $c^n_{i}$ converges to weak solution and it is derived from a series of step functions 
$c^h$. We will achieve this through various pointwise approximations as described below
 \\
	Midpoint approximation:	
	\begin{align*}
		X^h:\epsilon\in (0,R)\rightarrow
		X^h(\epsilon)=\sum_{i=1}^{\mathrm{I}}\epsilon_i\chi_{\Lambda_i^h}(\epsilon).
		\end{align*}
		Left endpoint approximation:
		\begin{align*}
		\xi^h:\epsilon\in (0,R)\rightarrow
		\xi^h(\epsilon)=\sum_{i=1}^{\mathrm{I}}\epsilon_{i-1/2}\chi_{\Lambda_i^h}(\epsilon).
		\end{align*}
	The subsequent lemma is an useful technique for achieving convergence.	
\begin{lem}{\label{Wconverge}} [\cite{laurenccot2002continuous}, Lemma A.2]
Let $\Omega^{'}$ be an open subset of $\mathbb{R}^m$ and let there exists a constant $M>0$ and two sequences $(q_n)_{n\in \mathbb{N}}$ and
$(s_n)_{n\in \mathbb{N}}$ such that $(q_n)\in L^1(\Omega^{'}), q\in L^1(\Omega^{'})$ and $$q_n\rightharpoonup q,\ \ \ \text{weakly in}\
\,L^1(\Omega^{'})\ \text{as}\ n\rightarrow \infty,$$ $(s_n)\in
L^\infty(\Omega^{'}), s \in L^\infty(\Omega^{'}),$ and for all $n\in
\mathbb{N}, |s_n|\leq M$ with $$s_n\rightarrow s,\ \ \text{almost
everywhere (a.e.) in}\ \ \Omega^{'}  \ \text{as}\ \ n\rightarrow
\infty.$$ Then
$$\lim_{n\rightarrow \infty}\|q_n(s_n-s)\|_{L^1(\Omega^{'})}=0$$
and $$q_n\, s_n\rightharpoonup q\, s,\ \ \ \text{weakly in}\
\,L^1(\Omega^{'})\ \text{as}\ n\rightarrow \infty.$$
\end{lem}\enter		
We have now gathered all the necessary evidence to prove Theorem \ref{maintheorem}. Let us consider  a test function $\varphi\in C^1([0,T]\times ]0,R])$, with compact support in $\varsigma$ within $[0,\varsigma_{N-1}]$ for small $\varsigma$. We determine the left endpoint approximation for the spatial variable and the finite volume for the time variable of $\varphi$ on $\varsigma_n\times \Lambda_i^h$ by
$$\varphi_i^n=\frac{1}{\Delta \varsigma}\int_{\varsigma _n}^{\varsigma_{n+1}}\varphi(\varsigma,\epsilon_{i-1/2})d\varsigma.$$  
 Multiply Eq.(\ref{fully}) by $\varphi_i^n$ and employing the summation with regard to $n$ \& $i$ lead to
	\begin{align}\label{convergence1}
	\sum_{n=0}^{N-1}\sum_{i=1}^{\mathrm{I}}\Delta \epsilon_{i} (c_{i}^{n+1}-c_{i}^{n})\varphi_i^n = &{\Delta \varsigma}\sum_{n=0}^{N-1}\sum_{i=1}^{\mathrm{I}}\sum_{l=1}^{\mathrm{I}}\sum_{j=i}^{\mathrm{I}}K_{j,l}c_{j}^{n}c_{l}^{n}\Delta \epsilon_{j}\Delta \epsilon_{l} \varphi_i^n\int_{\epsilon_{i-1/2}}^{\epsilon_{i+1/2}^{j}}b(\epsilon,\epsilon_{j},\epsilon_{l})\,d\epsilon\nonumber \\
	-&{\Delta \varsigma}\sum_{n=0}^{N-1}\sum_{i=1}^{\mathrm{I}}\sum_{j=1}^{\mathrm{I}}K_{i,j}c_{i}^{n}c_{j}^{n}\Delta \epsilon_{i}\Delta \epsilon_{j}\varphi_i^n\frac{\sum_{l=1}^{i}\epsilon_l\int_{\epsilon_{l-1/2}}^{\epsilon_{l+1/2}^{i}}b(\epsilon,\epsilon_{i},\epsilon_{j})d\epsilon}{\epsilon_i}. 		
	\end{align}
To observe the computation, we focus on the first term on the RHS of Eq.(\ref{convergence1})
	\begin{align}\label{convergence2}
	{\Delta \varsigma}\sum_{n=0}^{N-1}\sum_{i=1}^{\mathrm{I}}\sum_{l=1}^{\mathrm{I}}\sum_{j=i}^{\mathrm{I}}K_{j,l}c_{j}^{n}c_{l}^{n}\Delta \epsilon_{j}\Delta \epsilon_{l} \varphi_i^n\int_{\epsilon_{i-1/2}}^{\epsilon_{i+1/2}^{j}}b(\epsilon,\epsilon_{j},\epsilon_{l})\,d\epsilon \nonumber \\
	={\Delta \varsigma}\sum_{n=0}^{N-1}\sum_{i=1}^{\mathrm{I}}\sum_{l=1}^{\mathrm{I}}K_{i,l}c_{i}^{n}c_{l}^{n}\Delta \epsilon_{i} \Delta \epsilon_{l}\varphi_i^n &\int_{\epsilon_{i-1/2}}^{\epsilon_{i}}b(\epsilon,\epsilon_{i},\epsilon_{l})d\epsilon \nonumber \\
	 + \Delta \varsigma\sum_{n=0}^{N-1}\sum_{i=1}^{\mathrm{I}}\sum_{l=1}^{\mathrm{I}}\sum_{j=i+1}^{\mathrm{I}} K_{j,l}c_{j}^{n}c_{l}^{n}\Delta \epsilon_{j}& \Delta \epsilon_{l}\varphi_i^n \int_{\epsilon_{i-1/2}}^{\epsilon_{i+1/2}}b(\epsilon,\epsilon_{j},\epsilon_{l})d\epsilon.
	\end{align}
	Now, take the second term of Eq.(\ref{convergence1}) and open for the index $l$, to get two terms as
	\begin{align}\label{convergence21}
	{\Delta \varsigma}\sum_{n=0}^{N-1}\sum_{i=1}^{\mathrm{I}}\sum_{j=1}^{\mathrm{I}}K_{i,j}c_{i}^{n}c_{j}^{n}\Delta \epsilon_{i}\Delta \epsilon_{j}\varphi_i^n\frac{\sum_{l=1}^{i}\epsilon_l\int_{\epsilon_{l-1/2}}^{\epsilon_{l+1/2}^{i}}b(\epsilon,\epsilon_{i},\epsilon_{j})d\epsilon}{\epsilon_i}\nonumber \\
	={\Delta \varsigma}\sum_{n=0}^{N-1}\sum_{i=1}^{\mathrm{I}}\sum_{j=1}^{\mathrm{I}}K_{i,j}c_{i}^{n}c_{j}^{n}\Delta \epsilon_{i}\Delta \epsilon_{j}\varphi_i^n\int_{\epsilon_{l-1/2}}^{\epsilon_{i}}b(\epsilon,\epsilon_{i},\epsilon_{j})&d\epsilon \nonumber \\
	+{\Delta \varsigma}\sum_{n=0}^{N-1}\sum_{i=1}^{\mathrm{I}}\sum_{j=1}^{\mathrm{I}}K_{i,j}c_{i}^{n}c_{j}^{n}\Delta \epsilon_{i}\Delta \epsilon_{j}\varphi_i^n&\frac{\sum_{l=1}^{i-1}\epsilon_l \int_{\epsilon_{l-1/2}}^{\epsilon_{l+1/2}}b(\epsilon,\epsilon_{i},\epsilon_{j})d\epsilon}{\epsilon_i}.
	\end{align}
	Substituing Eqs.(\ref{convergence2}-\ref{convergence21}) in Eq.(\ref{convergence1}) yields the following equation
	\begin{align}\label{convergence12}
		\sum_{n=0}^{N-1}\sum_{i=1}^{\mathrm{I}}\Delta \epsilon_{i} (c_{i}^{n+1}-c_{i}^{n})\varphi_i^n = &\Delta \varsigma\sum_{n=0}^{N-1}\sum_{i=1}^{\mathrm{I}}\sum_{l=1}^{\mathrm{I}}\sum_{j=i+1}^{\mathrm{I}} K_{j,l}c_{j}^{n}c_{l}^{n}\Delta \epsilon_{j} \Delta \epsilon_{l}\varphi_i^n \int_{\epsilon_{i-1/2}}^{\epsilon_{i+1/2}}b(\epsilon,\epsilon_{j},\epsilon_{l})d\epsilon\nonumber \\
		-{\Delta \varsigma}&\sum_{n=0}^{N-1}\sum_{i=1}^{\mathrm{I}}\sum_{j=1}^{\mathrm{I}}K_{i,j}c_{i}^{n}c_{j}^{n}\Delta \epsilon_{i}\Delta \epsilon_{j}\varphi_i^n\frac{\sum_{l=1}^{i-1}\epsilon_l \int_{\epsilon_{l-1/2}}^{\epsilon_{l+1/2}}b(\epsilon,\epsilon_{i},\epsilon_{j})d\epsilon}{\epsilon_i}
	\end{align}
 For each $i$  and $n$, the Left-hand side (LHS) of  Eq.(\ref{convergence12}) resembles like this
	\begin{align*}
	\sum_{n=0}^{N-1}\sum_{i=1}^{\mathrm{I}}\Delta \epsilon_{i} (c_{i}^{n+1}-c_{i}^{n})\varphi_i^n = 	\sum_{n=0}^{N-1}\sum_{i=1}^{\mathrm{I}}\Delta \epsilon_{i} c_i^{n+1}(\varphi_i^{n+1}-\varphi_i^n)+ \sum_{i=1}^{\mathrm{I}}\Delta \epsilon_{i}   c_i^{in}\varphi_i^0.
	\end{align*}
Moreover, evaluating the aforementioned equation using the function $c^h$ yields		  	
\begin{align*}
\sum_{i=1}^{\mathrm{I}}\Delta \epsilon_{i} (c_{i}^{n+1}-c_{i}^{n})\varphi_i^n= &\sum_{n=0}^{N-1}\sum_{i=1}^{\mathrm{I}}\int_{t_{n+1}}\int_{\Lambda_i^h}c^h(\varsigma,\epsilon)\frac{\varphi(\varsigma,\xi^h(\epsilon))-\varphi(\varsigma-\Delta \varsigma,\xi^h(\epsilon))}{\Delta \varsigma}d\epsilon\,d\varsigma \\&+\sum_{i=1}^{\mathrm{I}}\int_{\Lambda_i^h}c^h(0,\epsilon)\frac{1}{\Delta \varsigma}\int_0^{\Delta \varsigma}\varphi(\varsigma,\xi^h(\epsilon))d\varsigma\,d\epsilon\\
= &\int_{\Delta \varsigma}^T \int_{0}^{R} c^h(\varsigma,\epsilon)\frac{\varphi(\varsigma,\xi^h(\epsilon))-\varphi(\varsigma-\Delta \varsigma,\xi^h(\epsilon))}{\Delta \varsigma}d\epsilon\,dt\\
& + \int_{0}^R c^h(0,\epsilon)\frac{1}{\Delta \varsigma}\int_0^{\Delta \varsigma}\varphi(\varsigma,\xi^h(\epsilon))d\varsigma\,d\epsilon.
\end{align*}	  
With the support of Lemma \ref{Wconverge}, and  we know that $c^h(0,\epsilon)\rightarrow c^{in}$ in $L^1]0,R[$, the following outcome	is certain		
	\begin{align}\label{finaltime1}
		\int_{0}^R c^h(0,\epsilon)\frac{1}{\Delta \varsigma}\int_0^{\Delta \varsigma}\varphi(\varsigma,\xi^h(\epsilon))d\varsigma d\epsilon\rightarrow \int_{0}^R c^{in}(\epsilon)\varphi(0,\epsilon)d\epsilon
	\end{align}
as max$\{h,\Delta \varsigma\}$ goes to $0$. Since, $\varphi\in C^1([0,T]\times ]0,R])$ has compact support and a bounded derivative, applying Taylor series expansion of $\varphi$, Lemma \ref{Wconverge} and Proposition \ref{equiintegrability} ensure that for max$\{h,\Delta \varsigma\} \rightarrow 0$ 
\begin{align*}
\int_0^T\int_0^R c^h(\varsigma,\epsilon)\frac{\varphi(\varsigma,\xi^h(\epsilon))-\varphi(\varsigma-\Delta
\varsigma,\xi^h(\epsilon))}{\Delta \varsigma}d\epsilon\,d\varsigma \rightarrow \int_0^T\int_0^R c(\varsigma,\epsilon)&\frac{\partial \varphi}{\partial \varsigma}(\varsigma,\epsilon)d\epsilon\,d\varsigma,
\end{align*}
	and hence, we obtain
	\begin{align}\label{finaltime2}
	&\int_{\Delta \varsigma}^T \int_0^R
	\underbrace{c^h(\varsigma,\epsilon)\frac{\varphi(\varsigma,\xi^h(\epsilon))-\varphi(\varsigma-\Delta \varsigma,\xi^h(\epsilon))}{\Delta \varsigma}}_{c(\varphi)} d\epsilon\,d\varsigma \nonumber \\ 
	&= \int_0^T \int_0^R c(\varphi)\, d\epsilon\,d\varsigma -\int_0^{\Delta \varsigma} \int_0^R c(\varphi)\, d\epsilon\,d\varsigma  \rightarrow \int_0^T \int_0^R  c(\varsigma,\epsilon)\frac{\partial \varphi}{\partial \varsigma}(\varsigma,\epsilon)d\epsilon\,d\varsigma.
	\end{align}
Next, evaluate the first term on the RHS of Eq.(\ref{convergence12}) that leads to
\begin{align}\label{convergence4}
\Delta \varsigma\sum_{n=0}^{N-1}\sum_{i=1}^{\mathrm{I}}\sum_{l=1}^{\mathrm{I}}\sum_{j=i+1}^{\mathrm{I}} K_{j,l}c_{j}^{n}c_{l}^{n}\Delta \epsilon_{j} \Delta \epsilon_{l}\varphi_i^n \int_{\epsilon_{i-1/2}}^{\epsilon_{i+1/2}}b(\epsilon,\epsilon_{j},\epsilon_{l})d\epsilon \nonumber \\
=\sum_{n=0}^{N-1}\sum_{i=1}^{\mathrm{I}}\sum_{l=1}^{\mathrm{I}}\sum_{j=i+1}^{\mathrm{I}}\int_{t_{n}}\int_{\Lambda_{i}^{h}}\int_{\Lambda_{l}^{h}}\int_{\Lambda_{j}^{h}}\Big\{K^h(\rho,\sigma)c^h(\varsigma,\rho)c^h(\varsigma,\sigma)\varphi(\varsigma,\xi^h(\epsilon))& \nonumber \\ \frac{1}{\Delta \epsilon_{i}}\int_{\Lambda_{i}^{h}}b(s,X^{h}(\rho),X^{h}(\sigma))ds\Big\}\,d\rho\,d\sigma\,d\epsilon\,d\varsigma \nonumber \\
= \int_{0}^{T}\int_{0}^{R}\int_{0}^{R}\int_{\Xi^{h}(\epsilon)}^{R} K^h(\rho,\sigma)c^h(\varsigma,\rho)c^h(\varsigma,\sigma)\varphi(\varsigma,\xi^h(\epsilon))b(X^{h}(\epsilon),X^{h}(\rho),X^{h}(\sigma))&d\rho\,d\sigma\,d\epsilon\,d\varsigma.
\end{align}
According to Proposition \ref{equiintegrability}, Lemma \ref{Wconverge}, and Equation (\ref{convergence4}), as max $\{h,\Delta \varsigma\} \rightarrow 0$,
\begin{align}\label{convergence5}
\Delta \varsigma\sum_{n=0}^{N-1}\sum_{i=1}^{\mathrm{I}}\sum_{l=1}^{\mathrm{I}}\sum_{j=i+1}^{\mathrm{I}} K_{j,l}c_{j}^{n}c_{l}^{n}\Delta \epsilon_{j} \Delta \epsilon_{l}\varphi_i^n \int_{\epsilon_{i-1/2}}^{\epsilon_{i+1/2}}b(\epsilon,\epsilon_{j},\epsilon_{l})d\epsilon \nonumber \\
\rightarrow \int_{0}^{T}\int_{0}^{R}\int_{0}^{R}\int_{\epsilon}^{R}\varphi(\varsigma,\epsilon)K(\rho,\sigma)b(\epsilon,\rho,\sigma)&c(\varsigma,\rho)c(\varsigma,\sigma)d\rho\,d\sigma\,d\epsilon\,d\varsigma.
\end{align}
Now, to simplify the second term on the RHS of Eq.(\ref{convergence12}), as
	\begin{align}\label{convergence6}
	{\Delta \varsigma}\sum_{n=0}^{N-1}\sum_{i=1}^{\mathrm{I}}\sum_{j=1}^{\mathrm{I}}K_{i,j}c_{i}^{n}c_{j}^{n}\Delta \epsilon_{i}\Delta \epsilon_{j}\varphi_i^n \frac{\sum_{l=1}^{i-1}\epsilon_l \int_{\epsilon_{l-1/2}}^{\epsilon_{l+1/2}}b(\epsilon,\epsilon_{i},\epsilon_{j})d\epsilon}{\epsilon_i}
		=\sum_{n=0}^{N-1}\sum_{i=1}^{\mathrm{I}}\sum_{j=1}^{\mathrm{I}}\int_{t_{n}}\int_{\Lambda_i^h}\int_{\Lambda_j^h}\Big\{K^h(\epsilon,\rho)c^h(\varsigma,\epsilon)\nonumber \\c^h(\varsigma,\rho)\varphi(\varsigma,\xi^h(\epsilon))
		\frac{\sum_{l=1}^{i-1}\int_{\Lambda_l^h}X^{h}(\sigma)
		\frac{1}{\Delta \epsilon_{l}}\int_{\Lambda_{l}^{h}}b(s,X^{h}(\epsilon),X^{h}(\rho))ds d\sigma}{X^{h}(\epsilon)}\Big\}d\rho\,d\epsilon\,d\varsigma \nonumber \\
			=\sum_{n=0}^{N-1}\sum_{i=1}^{\mathrm{I}}\sum_{j=1}^{\mathrm{I}}\int_{t_{n}}\int_{\Lambda_i^h}\int_{\Lambda_j^h}K^h(\epsilon,\rho)c^h(\varsigma,\epsilon)c^h(\varsigma,\rho)\varphi(\varsigma,\xi^h(\epsilon))\frac{\int_{0}^{\xi^{h}(\epsilon)}X^{h}(\sigma)b(X^{h}(\sigma),X^{h}(\epsilon),X^{h}(\rho))d\sigma}{X^{h}(\epsilon)}d\rho\,d\epsilon\,d\varsigma& 
		\end{align}
	as  max$\{h,\Delta \varsigma\} \rightarrow 0$ and property (\ref{breakageconservation}), above equation follows
	\begin{align}\label{convergence7}
	{\Delta \varsigma}\sum_{n=0}^{N-1}\sum_{i=1}^{\mathrm{I}}\sum_{j=1}^{\mathrm{I}}K_{i,j}c_{i}^{n}c_{j}^{n}\Delta \epsilon_{i}\Delta \epsilon_{j}\varphi_i^n\frac{\sum_{l=1}^{i-1}\epsilon_l \int_{\epsilon_{l-1/2}}^{\epsilon_{l+1/2}}b(\epsilon,\epsilon_{i},\epsilon_{j})d\epsilon}{\epsilon_i}\nonumber \\
		\rightarrow\int_{0}^{T}\int_{0}^{R}\int_{0}^{R}K(\epsilon,\rho)c(\varsigma,\epsilon)c(\varsigma,\rho)\varphi(\varsigma,\epsilon)\frac{\int_{0}^{\epsilon}\sigma b(\sigma,\epsilon,\rho)d\sigma}{\epsilon}d\rho\,d\epsilon\,d\varsigma \nonumber \\
= \int_{0}^{T}\int_{0}^{R}\int_{0}^{R}K(\epsilon,\rho)c(\varsigma,\epsilon)c(\varsigma,\rho)\varphi(\varsigma,\epsilon)d\rho\,d\epsilon\,d\varsigma.
	\end{align}
 Eqs.(\ref{convergence12})-(\ref{convergence7}) provide the expected outcomes for the weak convergence, as seen in Eq.(\ref{convergence0}).

\section{Conclusions}\label{conclusion}  

This article presented a theoretical analysis of the convergence properties of the FVM utilized in solving the CBE on non-uniform meshes in the weighted $L^1$ space. The analysis focused on a conservative scheme by introducing a weight function, considering scenarios with unbounded collision kernels and singular breakage distribution kernels. Employing the weak 
$L^1$ compactness method containing equiboundedness and equiintegrability, and non-negativity of the approximated solution of the problem with comprehensive investigation into weak convergence was conducted. 
					
\section{Acknowledgments}
The first author acknowledges the support of CSIR India, which provides the PhD fellowship under file No. 1157/CSIR-UGC NET June 2019. Rajesh Kumar expresses gratitude to the Science and Engineering Research Board (SERB), DST India, for funding through project MTR/2021/000866.
		
\section*{Conflict of Interest} This work does not have any conflicts of interest.

\bibliography{colref}
\bibliographystyle{ieeetr}
				\end{document}